\documentclass{amsart}
\usepackage{amsthm,amssymb,amsmath,color}
\usepackage{pdfsync}
\usepackage{graphicx}

\def\E{\mathbb E}
\def\R{\mathbb R}

\def\N{\mathbb N}
\def\P{\mathbb P}

\def\Z{\mathbb Z}

\def\bfB{\mathbf B}

\def\cF{{\mathcal F}}

\newtheorem{Def}{Definition}
\newtheorem{Prop}{Proposition}
\newtheorem{lem}{Lemma}
\newtheorem{Theo}{Theorem}

\begin{document}

\title[Stable QML estimation of the EGARCH(1,1)]{Continuous invertibility and stable QML estimation of the EGARCH(1,1) model}

\author[O. Wintenberger]{Olivier Wintenberger}

\address{Olivier Wintenberger,
            Universit\'e de Paris-Dauphine,
            Centre De Recherche en Math\'ematiques de la D\'ecision
            UMR CNRS 7534,
            Place du Mar\'echal De Lattre De Tassigny,
            75775 Paris Cedex 16, France}
            \email{owintenb@ceremade.dauphine.fr}

\maketitle

\begin{abstract}
We introduce the  notion of continuous  invertibility on a compact set for volatility models driven by a Stochastic Recurrence Equation (SRE). We prove the strong consistency of the Quasi Maximum Likelihood Estimator (QMLE) when the optimization  procedure is done on a   continuously invertible domain. This approach gives for the first time the strong consistency of the QMLE used by Nelson in \cite{nelson:1991} for the EGARCH(1,1) model under explicit but non observable conditions. In practice, we propose to stabilize the QMLE by constraining the optimization procedure to an empirical continuously invertible domain. The new method, called Stable QMLE (SQMLE), is strongly consistent when the observations follow an invertible EGARCH(1,1) model. We also give the asymptotic normality of the SQMLE  under additional minimal assumptions. \end{abstract}



\section{Introduction}\label{sec:intro}

Since the seminal papers \cite{engle:1982,bollerslev:1986}, the General Autoregressive Conditional Heteroskedasticity  (GARCH) type models have been successfully applied to financial time series modeling. One of the stylized facts observed on the data is the asymmetry with respect to (wrt) shocks \cite{cont:2001}: a negative past observation impacts the present volatility more importantly than a positive one. Nelson  introduced in \cite{nelson:1991} the Exponential-GARCH (EGARCH) model that reproduces this asymmetric effect. Not surprisingly, theoretical investigations of EGARCH has attracted lot of attention since then, see for example \cite{he:terasvirta:malmsten:2002,mikosch:rezapour:2012}. However, the properties of the Quasi Maximum Likelihood Estimator (QMLE)  used empirically in \cite{nelson:1991} was not proved except in some degenerate case, see \cite{straumann:mikosch:2006}. We give in this paper some sufficient conditions for the strong consistency and the asymptotic normality of the QMLE in the EGARCH(1,1) model. Our approach of the strong consistency is based on the natural notion of continuous invertibility that we introduce in the general setting of volatility model solutions of a Stochastic Recurrent Equation (SRE).\\

Consider a real valued volatility model of the form $ X_{t} =\sigma_tZ_{t}$ where $\sigma_{t}$ is the volatility and where the innovations $Z_t$ are normalized, centered independent identical distributed (iid) random vectors.  It is assumed that a transformation of the volatility satisfies some parametric SRE  (also called Iterated Random Function): there exist a function $h$  and some  $\psi_{t}$ measurable wrt $Z_t$ such that the following relations 
\begin{equation}\label{eq:gmod}
h(\sigma_{t+1}^2)  = \psi_{t}(h(\sigma_t^2),\theta_0),\qquad \forall t\in \Z
\end{equation} hold. Classical examples are the GARCH(1,1) and  EGARCH(1,1) models:
\begin{align}
\mbox{GARCH(1,1):}& \qquad\sigma_{t+1}^2=\alpha_0+\beta_0 \sigma_{t}^2+\gamma_0 X_{t}^2,\label{eq:garch}\\
\mbox{EGARCH(1,1):}&\qquad  \log(\sigma_{t+1}^2)=\alpha_0+\beta_0 \log(\sigma_{t}^2)+(\gamma_0 Z_{t}+\delta_0|Z_{t}|).\label{eq:egarch}
\end{align}

In practice, the innovations $Z_t$ are not observed. Writing $Z_t= X_t/\sigma_t$ in the expression of $\psi_t$ wrt $Z_t$, we   invert the model, i.e we consider a new SRE driven by a function $\phi_t$  of the observation $X_t$,
\begin{equation}\label{SREinftrue}
h(\sigma_{t+1}^2)=\phi_t(h(\sigma_t^2),\theta_0) ,\qquad t\in\Z.
\end{equation}
 For instance, we obtain from \eqref{eq:garch} the inverted model $\phi_t(x,\theta )=\alpha+\beta x+\gamma X_{t}^2$ for the GARCH(1,1) model. For the EGARCH(1,1) model, we obtain from \eqref{eq:egarch} the inverted model
\begin{equation}\label{eq:egarchinv}
\phi_t(x,\theta)=\alpha+\beta x+(\gamma  X_{t}+\delta |X_{t}|)\exp(-x/2).
\end{equation}
In accordance with the notions of invertibility given in \cite{granger:andersen:1978,tong:1993,straumann:2005,straumann:mikosch:2006}, we will say that the model is invertible if the SRE \eqref{SREinftrue} is stable. Then, as  the functions $\phi_t$ are observed, the volatility is efficiently forecasted by using recursively the relation
\begin{equation}\label{sre}
h_{t+1}=\phi_t(h_t,\theta_0),\qquad t\ge0,
\end{equation}
from an arbitrary initial value $h_0$. Sufficient conditions for the convergence of this SRE are the negativity of a Lyapunov coefficient and the existence of logarithmic moments, see \cite{bougerol:1993}. So the  GARCH(1,1) model is invertible as soon as $0\le \beta_0<1$. The invertibility of the EGARCH(1,1) model is more complicated to assert due to  the exponential function in \eqref{eq:egarchinv}. The recursive relation on $h_{t+1}$ can explode to $-\infty$ for small negative values of $h_t$ and negative values of $\gamma_0  X_{t}+\delta_0 |X_{t}|$. However, assuming that $\delta_0\ge|\gamma_0|$, the relation $\gamma_0  X_{t}+\delta_0 |X_{t}|>0$ holds and conditions for invertibility  of the EGARCH(1,1), denoted hereafter INV($\theta_0)$, are obtained in \cite{straumann:2005,straumann:mikosch:2006}:\\

INV($\theta_0$): $\delta_0\ge |\gamma_0|$ and 
\begin{equation}
\E[\log(\max\{ \beta_0 , 2^{-1}(\gamma_0  X_{0}+\delta_0 |X_{0}|)\exp(-2^{-1}\alpha_0 /(1-\beta_0 ))-\beta_0 \})]<0.\label{eq:condtheta}
\end{equation}
 On the oppposite, Sorokin introduces  in \cite{sorokin:2011} sufficient conditions on $\theta_0$ for the EGARCH(1,1) model to be non-invertible. Then the SRE \eqref{eq:egarchinv} is completely chaotic for any possible choice of the initial value $h_0$ and the volatility forecasting procedure  based on the model \eqref{SREinftrue} is not reliable.\\

In practice, the value $\theta_0=(\alpha_0,\beta_0,\gamma_0,\delta_0)$ of the Data Generating Process (DGP) EGARCH(1,1) is unknown. Nelson proposed in \cite{nelson:1991} to estimate $\theta_0$ with the Quasi Maximum Likelihood Estimator (QMLE). Let us recall the definition of this classical estimator that estimates efficiently many GARCH models, see \cite{berkes:horvath:kokoszka:2003} for the GARCH case and \cite{francq:zakoian:2004} for the ARMA-GARCH case. To construct the QMLE one approximates the volatility using the observed SRE \eqref{sre} at any $\theta$. Assume that the SREs driven by $\phi_t(\cdot,\theta)$ are stable for any $\theta$. Let us consider the functions $\hat g_t(\theta)$ defined for any $\theta$ as the recursive solutions of the SRE 
\begin{equation}\label{SREgen}
\hat g_{t+1}(\theta)=\phi_t(\hat g_t(\theta),\theta),\qquad t\ge0,
\end{equation}
 for some arbitrary initial value $\hat g_0(\theta)$. Assume that $h$ is a bijective function of inverse $\ell>0$. Then $\ell(\hat g_t(\theta_0))$ has a limiting law that coincides with the one of $ \sigma_t^2$. The  Quasi Likelihood (QL) criteria is defined as
\begin{equation}\label{eq:ql}
2 n\hat{L}_{n}(\theta)  =  \sum_{t=1}^{n}\hat{l}_{t}(\theta) 
  =  \sum_{t=1}^{n} X_{t}^{2}/\ell(\hat{g}_{t}(\theta))^{2}+\log(\ell(\hat{g}_{t}(\theta)).
\end{equation}
The associated $M$-estimator is the QMLE $\hat{\theta}_{n}$ defined by optimizing the QL on some compact set $\Theta$
$$
\hat \theta_{n}=\mbox{argmin}_{\theta\in\Theta}\hat{L}_{n}(\theta).
$$
This estimator has been used since the seminal paper of Nelson \cite{nelson:1991} for estimating the EGARCH(1,1) model without any theoretical justification, see for example \cite{brandt:jones:2006}. The inverted EGARCH(1,1) model is driven by the SRE \eqref{SREgen} that expresses as (denoting $\ell(\hat{g}_{t})=\hat\sigma_{t}^2$)
\begin{equation}\label{egarchinv}
\log(\hat\sigma_{t+1}^2(\theta))=\alpha+\beta \log(\hat\sigma_{t}^2(\theta))+(\gamma  X_{t}+\delta |X_{t}|)\exp(-\log(\hat\sigma_{t}^2(\theta))/2).
\end{equation}
The consistency and the asymptotic normality are not proved except in the degenerate case $\beta=0$ for all $\theta\in\Theta$ in \cite{straumann:2005}. The problem of the procedure (and of any volatility forecast) is that the inverted EGARCH(1,1) model \eqref{egarchinv} is stable only for some values of $\theta$. Thus, contrary to other GARCH models, the QML estimation procedure is not always reliable for the EGARCH model, see the discussion in \cite{harvey:chakravarty:2008}. Thus, other estimation procedure has been investigated such as the bayesian, bias correction and the Whittle procedure in \cite{vrontos:dellaportas:politis:2000,demos:kyrikopoulou:2010,zaffaroni:2009} respectively. Another approach is to introduce models that behave like the EGARCH(1,1) model but where the QMLE could be more reliable, see \cite{harvey:chakravarty:2008,sucarrat:escribano:2010,francq:wintenberger:zakoian}.\\

We prove the strong consistency of the QMLE for the general model \eqref{eq:gmod} when the maximization procedure is done on a continuously invertible domain. We give sufficient conditions called the continuous invertibility of the model such that the QMLE is strongly consistent. More precisely we assume that the SRE \eqref{SREgen} produces continuous functions $\hat g_t$ of $\theta$ on $\Theta$. The continuous invertibility holds when the limiting law of $\hat g_t$ corresponds to the law of some continuous function $g_t$ on $\Theta$ that does not depend on the initial function $\hat g_0$. The continuous invertibility ensures the stability of the estimation procedure regardless the initial function $\hat g_0$ chosen arbitrarily in practice. Under few other assumptions, we prove that the QMLE is strongly consistent for continuously invertible models on the compact set $\Theta$. The continuous invertiblity should be checked systematically on models before using QMLE. One example of such continuously invertible models with properties similar than  the EGARCH model is the Log-GARCH model studied in \cite{francq:wintenberger:zakoian}. \\

As the continuous invertibility is an abstract assumption, we provide sufficient conditions for continuous invertibility  collected  in the assumption {\bf (CI)} below. These conditions ensure the invertibility of the model at any point $\theta$ of the compact set $\Theta$ and some regularity of the model with respect to the parameter $\theta$. As the inverted EGARCH(1,1) model \eqref{egarchinv} is a regular function of $\theta$, it satisfies {\bf (CI)} on any $\Theta$ such that the invertibility condition INV($\theta$) is satisfied for any $\theta\in \Theta$. Thus we prove the strong consistency of the QMLE for the invertible EGARCH(1,1) model  when INV($\theta$) is satisfied for any $\theta\in \Theta$. It is a serious advantage of our approach based on the continuous invertibility condition {\bf (CI)} compared with the approach of \cite{straumann:2005}. Based on uniform Lipschitz coefficients this last approach is more restrictive than our when applied to the EGARCH(1,1) model.  Moreover, we also prove the strong consistency of the natural volatility forecasting $\hat \sigma_n^2=\ell(\hat g_n(\hat\theta_n))$ of $\sigma_{n+1}^2$ under  {\bf (CI)}. Continuous invertibility seems to be well suited to assert volatility forecasting because $\hat \sigma_n^2$ expresses  as  functions $\hat g_n$ evaluated at points $\hat\theta_n\neq \theta_0$. To infer in practice the EGARCH(1,1) model, we propose to stabilize the QMLE. We constrain the QMLE on some compact set satisfying the empirical version of the condition INV($\theta_0$). This new estimator  $\hat \theta_n^S$ called Stable QMLE (SQMLE) produces only reliable volatility forecasting such that $\hat \sigma_n^2=\ell(\hat g_n(\hat\theta_n^S))$ does not depend asymptotically of the initial value $\hat\sigma_0^2=\ell(\hat g_0)$. It is not the case of the classical QMLE the continuous invertibility condition $INV(\theta)$ is not observed in practice. Thus INV($\theta$) might not be satisfied for any $\theta$ in the compact set $\Theta$ of the maximization procedure. And the whole procedure might have some chaotic behavior with respect to any initial value $\hat\sigma_0^2$ used in the inverted model.\\

The asymptotic normality of the SQMLE in the EGARCH(1,1) model is proved under the additional assumption {\bf (MM)}. The moment conditions in {\bf (MM)} are sufficient for the existence of the asymptotic covariance matrix. No uniform moment condition on the score vector is assumed. The proof is based on functional SREs as \eqref{sreg} and their perturbations. As for the strong consistency, we need to refine arguments from \cite{straumann:mikosch:2006}. As we do not apply any uniform Strong Law of Large Number, we develop in the proof new arguments based on SREs satisfied by differences such as $|\hat L_n(\hat\theta_n)-\hat L_n(\theta_0)|$. We consider conditions of moments {\bf (MM)} only at the point $\theta_0$ where   the expression of the score vector simplifies.  In the EGARCH(1,1) model, the conditions {\bf (MM)} take the simple form $\E[Z_0^4]<\infty$ and $\E[(\beta_0 -2^{-1}(\gamma_0 Z_0+\delta_0\left|Z_0\right|)^2]<1$ and can be checked in practice by estimating the innovations. We believe that this new approach gives sharp conditions for asymptotic normality for other models.\\

The paper is organized as follows. In Section \ref{sec:ci}, we discuss the standard notions of invertibility and introduce the  continuous invertibility and its sufficient condition {\bf (CI)}. We prove the strong consistency of the QMLE for general continuously invertible models in Section \ref{sec:sc}. The consistency of the volatility forecasting is also proved under the sufficient condition {\bf (CI)}. We apply this results in the EGARCH(1,1) model in Section \ref{sec:iqmle}. For this model, we propose a new method called Stable QMLE that produces only reliable volatility forecasting. The asymptotic normality of SQMLE for the EGARCH(1,1) model is given in Section \ref{sec:an}. The proofs of  technical Lemmas are collected in Section \ref{sec:pr}. 
 
\section{Preliminaries}\label{sec:ci}
\subsection{The general volatility  model}
In  this paper,  the innovations $Z_t\in \R $ are iid random variables (r.v.) such that $Z_t$ is centered and normalized, i.e. $\E[Z_0]=0$ and $\E[Z_0^2]=1 $. Consider the general DGP $
X_{t} = \sigma_t Z_{t}$ satisfying $
h(\sigma_{t+1}^2)  = \psi_{t}(h(\sigma_{ t}^2),\theta_{0})$ for all $t\in\Z$. The function $h$ is a bijection from some subset $\R^+$ to some subset of $\R$
of inverse $\ell$ called the link function. A first question regarding such general SRE is the existence of the model, i.e. wether or not a stationary solution exists.  Hereafter, we work under the general assumption
\begin{description}
\item[(ST)] The SRE \eqref{eq:gmod}   admits a unique stationary solution denoted $(\sigma_t^2)$ that is non anticipative, i.e. $\sigma_t^2$ is independent of $(Z_t,Z_{t+1},Z_{t+2},\ldots)$ for all $t\in\Z$, and  has finite log-moments: $\E\log^+ \sigma_0^2<\infty$. 
\end{description}

The GARCH(1,1) model \eqref{eq:garch} satisfies the condition {\bf (ST)} if and only if (iff) $\E[\log(\beta_0+\gamma_0Z_0^2)]<0$, see \cite{nelson:1990} for the existence of the stationary solution and \cite{berkes:horvath:kokoszka:2003} for the existence of log moments.
The EGARCH(1,1) model \eqref{eq:egarch} satisfies the condition {\bf (ST)}  iff $|\beta_0|<1$, see \cite{nelson:1991}. In this case, the model has nice ergodic properties: any process recursively defined by the SRE from an arbitrary initial value approximates exponentially fast a.s. the original process $(\sigma_t^2)$. In the sequel, we say that the sequence of non negative r.v. $(W_{t})$ converges exponentially almost surely to $0$,  
$W_{t}\xrightarrow{{e.a.s.}}0$ as $t\to \infty$, if  $W_{t} =o(e^{-Ct})$ a.s. for some r.v. $C>0$. We will also use the notation $x^+$ for the positive part of $x$, i.e. $x^+=x\vee 0$ for any $x\in\R$.

\subsection{Invertible models}\label{subsec:inv}

Under {\bf (ST)} the process $(X_t)$ is stationary, non anticipative and thus ergodic as a Bernoulli shift of an ergodic sequence $(Z_t)$, see \cite{krengel:1995}. Let us now investigate the question of invertibility of the general model \eqref{eq:gmod}. The classical notions of  invertibility are related with convergences of SRE and thus are implied by  Lyapunov conditions of Theorem 3.1 in \cite{bougerol:1993}.  Following \cite{tong:1993}, we say that a volatility model is invertible if the  volatility can be expressed as a function of the past observed values:
\begin{Def}\label{def:inv}
Under {\bf (ST)}, the model is invertible if the sequence of the volatilities $(\sigma_t^2 )$ is adapted to the filtration generated by $(X_{t-1},X_{t-2},\cdots)$.
\end{Def}
Using the relation $Z_t= X_t/\ell(h(\sigma_t))$ in the expression of $\psi_t$ yields the new SRE \eqref{SREinftrue}: $h(\sigma_{t+1}^2)=\phi_t(h(\sigma_t^2),\theta_0)$. Now the  random functions  $ \phi_t(\cdot,\theta_0) $ depends only on $X_t$. As $(X_t)$
is an ergodic and stationary process, it is also the case of the sequence of parametrized maps $(\phi_t(\cdot,\theta_0))$. Using Theorem 3.1 in \cite{bougerol:1993}, the invertibility of the model follows  if the $\phi_t(\cdot,\theta_0)$ are Lipschitz maps  such that there exists $r>0$ satisfying
\begin{equation} \label{eq:inv}
\mbox{inv}(\theta_0)\E[\log^+|\phi_0(x,\theta_0)|]<\infty \mbox{ for some }x\in E,\;\E[\log^+ \Lambda(\phi_0(\cdot,\theta_0) )]<\infty \mbox{ and } \E[\log \Lambda(\phi_0(\cdot,\theta_0)^{(r)})]<0.
\end{equation}
Here $\Lambda(f)$ denotes the Lipschitz coefficient of any function $f$ defined  by the relation (in the case where $f$ is real valued)
$$
\Lambda(f)=\sup_{x\neq y}\frac{|f(x)-f(y)|}{|x-y|}
$$
$f_t^{(r)}$ denotes the iterate $f_t\,o\,f_{t-1}\,o\cdots o\,f_{t-r}$ for any sequence of function $(f_t)$.
The conditions \eqref{eq:inv} are called the conditions of invertibility in \cite{straumann:mikosch:2006} and is  proved there that
\begin{Prop}
Under {\bf (ST)} and \eqref{eq:inv}, the general model \eqref{SREinftrue} is invertible.
\end{Prop}
The GARCH(1,1) model \eqref{eq:garch} is invertible as soon as $0\le \beta_0<1$.\\

The invertibility of the EGARCH(1,1) model is more difficult to assert due to the exponential function in the SRE \eqref{eq:egarchinv}. Let us describe the sufficient condition of invertibility INV($\theta_0$) of the EGARCH(1,1) model given in \cite{straumann:2005,straumann:mikosch:2006}. It expresses as a Lyapunov condition  \eqref{eq:condtheta} on the coefficients $\theta_0=(\alpha_0,\beta_0,\gamma_0,\delta_0)$. This condition does not depend on $\alpha_0$ when the DGP $(X_t)$ is itself the stationary solution of the EGARCH(1,1) model for $\theta_0$. Indeed, $(\log\sigma_t^2)$ admits a MA($\infty$) representation 
$$
\log\sigma_{t}^{2}=\alpha_0(1-\beta_0)^{-1}+\sum_{k=1}^{\infty}\beta_0^{k-1}(\gamma_0Z_{t-k}+\delta_0|Z_{t-k}|).
$$
Plugging in this MA($\infty$) representation into \eqref{eq:condtheta}, we obtain the equivalent sufficient condition
\begin{equation}
\E\Big[\log \Big(\max\Big \{\beta_0,2^{-1}\exp\Big(2^{-1}\sum_{k=0}^{\infty}\beta_0^{k}(\gamma_0 Z_{-k-1}+\delta_0 \left|Z_{-k-1}\right|)\Big)
 (\gamma_0 Z_{0}+\delta_0 \left|Z_{0}\right|)-\beta_0\Big\}\Big)\Big]<0.
\label{eq:condtheta2}
\end{equation}
Using the Monte Carlo algorithm and assuming that $Z_0 $
is $\mathcal N(0,1)$-distributed, we report in
Figure \ref{fig:Beta} the largest values of $\beta_0$ that satisfies the condition \eqref{eq:condtheta}
 on a grid of values of $(\gamma_0,\delta_0)$. 
 \begin{figure}[h!]
\centering
{\includegraphics[scale=0.4]{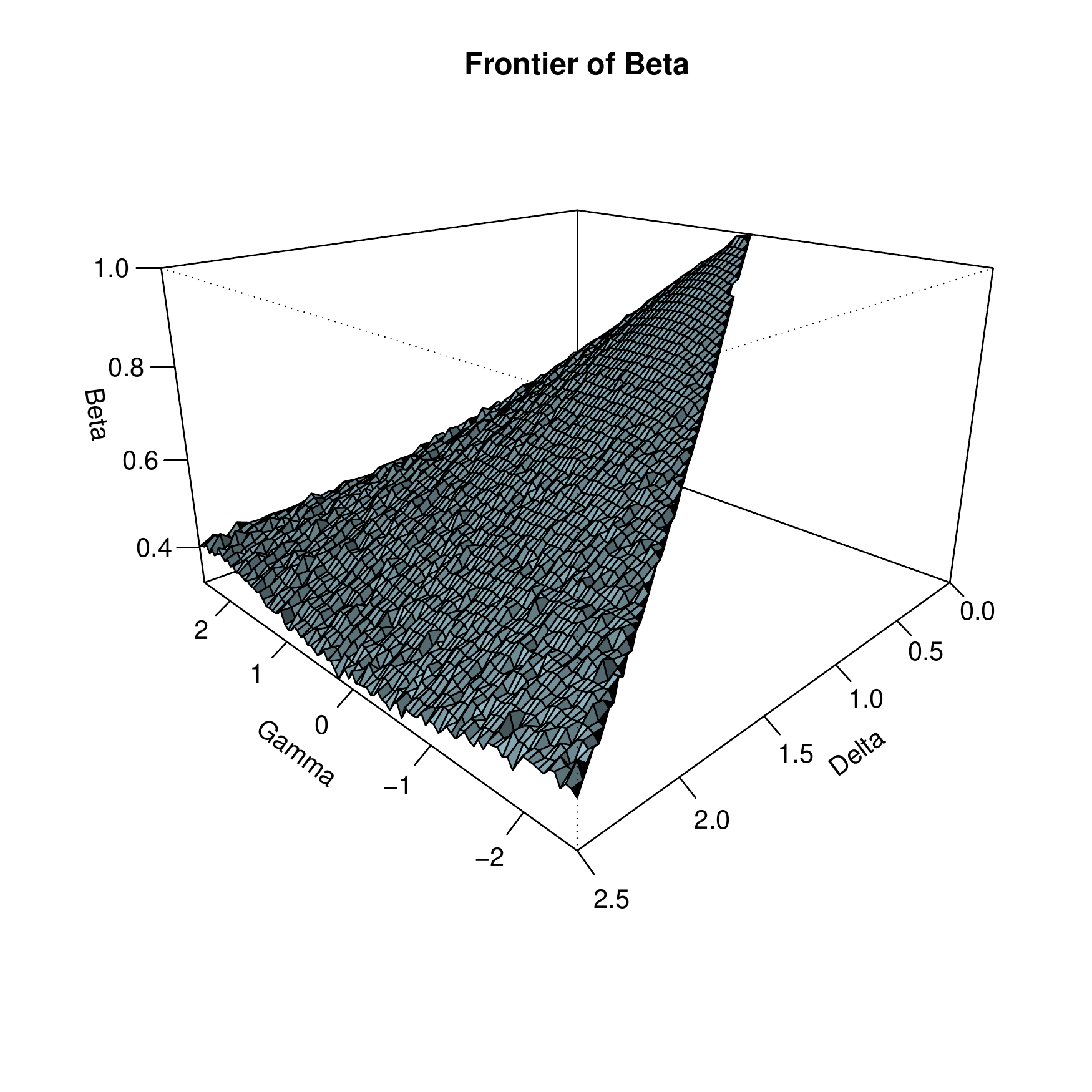} \includegraphics[scale=0.4]{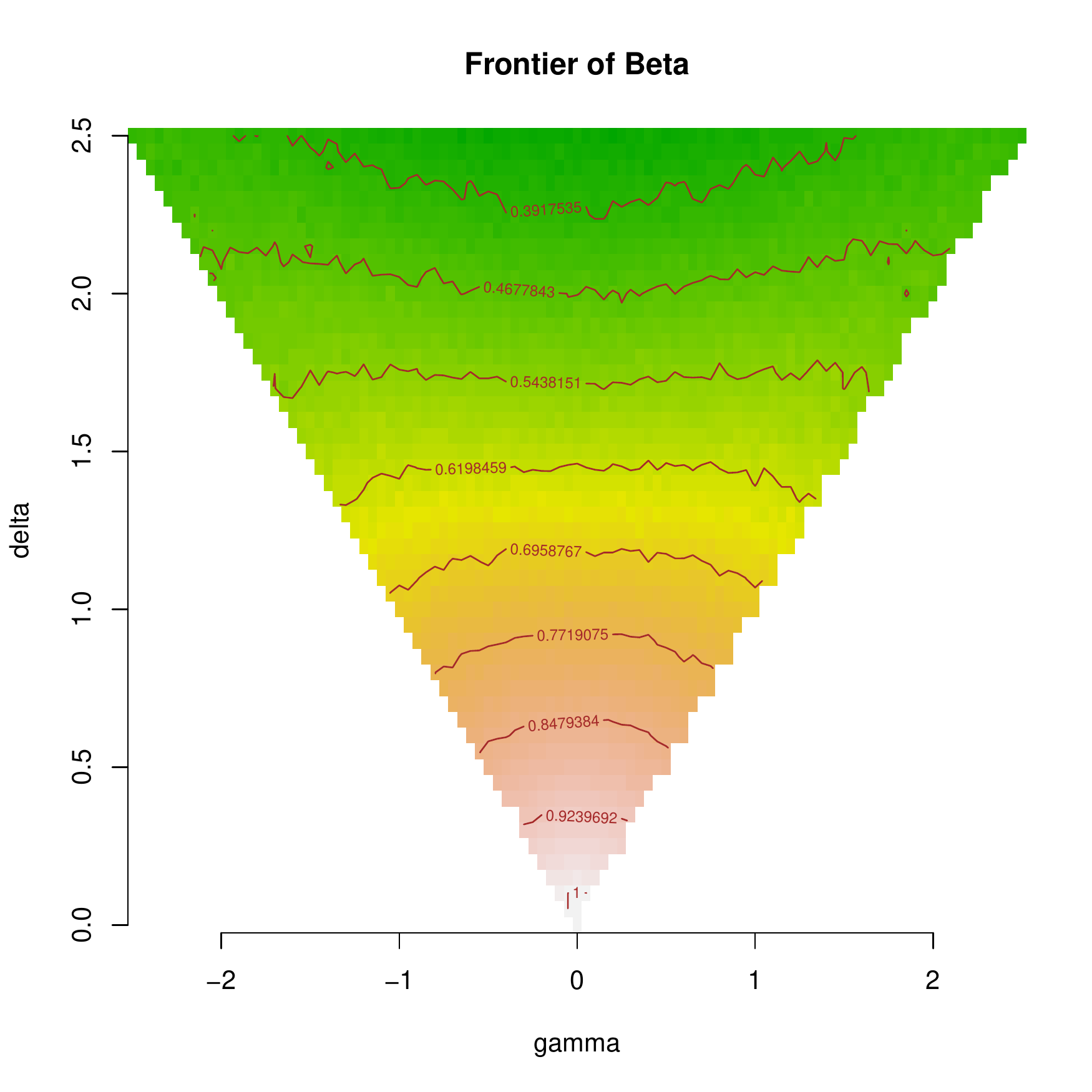}}
  \caption{Perspective and contour plots of the domain for invertibility.}
 \label{fig:Beta}
\end{figure}

The constraint on $\beta_0$ is always stronger than the  stationary constraint $|\beta_0|<1$. It exists stationary EGARCH(1,1) models that are not invertible, i.e. the inverted model   
$$
\log \hat\sigma_{t+1}^2=\alpha_0+\beta_0 \log \sigma_{t}^2+(\gamma_0  X_{t}+\delta_0 |X_{t}|)\exp(-\log \sigma_{t+1}^2/2),\qquad t\ge 0,
$$
is not stable wrt  any possible choice of initial value $\log\hat\sigma_0^2$. On the opposite, Sorokin exhibits in  \cite{sorokin:2011} sufficient conditions for some chaotic behaviour of the inverted EGARCH(1,1) model under $|\beta_0|<1$. To emphasize the danger to work with non invertible EGARCH(1,1) models, we report in Figure \ref{fig:conv} the convergence criterion $\sum_{t=1}^N(\log\sigma_t^2-\log\hat\sigma_t^2)$ wrt arbitrary initial values for two different values of $\theta_0$ ($N=10~000$). The first picture represents a stable case where INV($\theta_0$) holds. The second picture represents a chaotic case where the condition of non invertibility given in \cite{sorokin:2011} is satisfied. The stationary constraint $|\beta_0|<1$ is satisfied in both cases. It is interesting to note that the convergence criterion does not explode in the chaotic case. It is an important difference between the non invertible GARCH(1,1) and EGARCH(1,1) models: the non invertible GARCH(1,1) is always explosive because it is also non stationary. The inference by QMLE remains stable due to this very specific behaviour, see \cite{jensen:rahbek:2004}. On the opposite, we have driven numerical experiments and we are convinced that the QMLE procedure is not stable when the EGARCH(1,1) model is non invertible. It is not surprising as it is not possible to recover the volatility process from the inverted model,  even when the parameter $\theta_0$ of the DGP is known.

 \begin{figure}[h!]
\centering
{\begin{center}
\includegraphics[height=4cm,width=7cm]{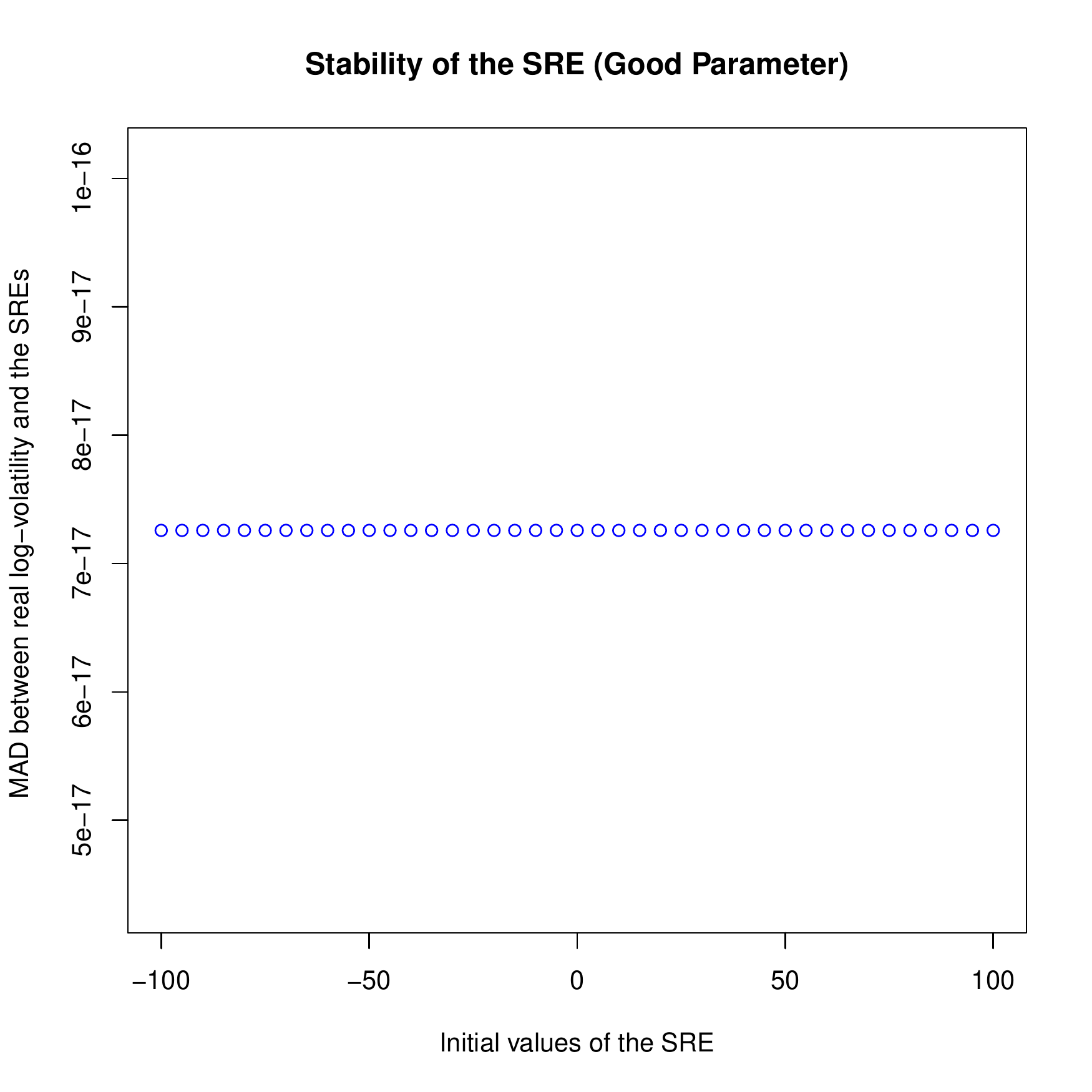}
\includegraphics[height=4cm,width=7cm]{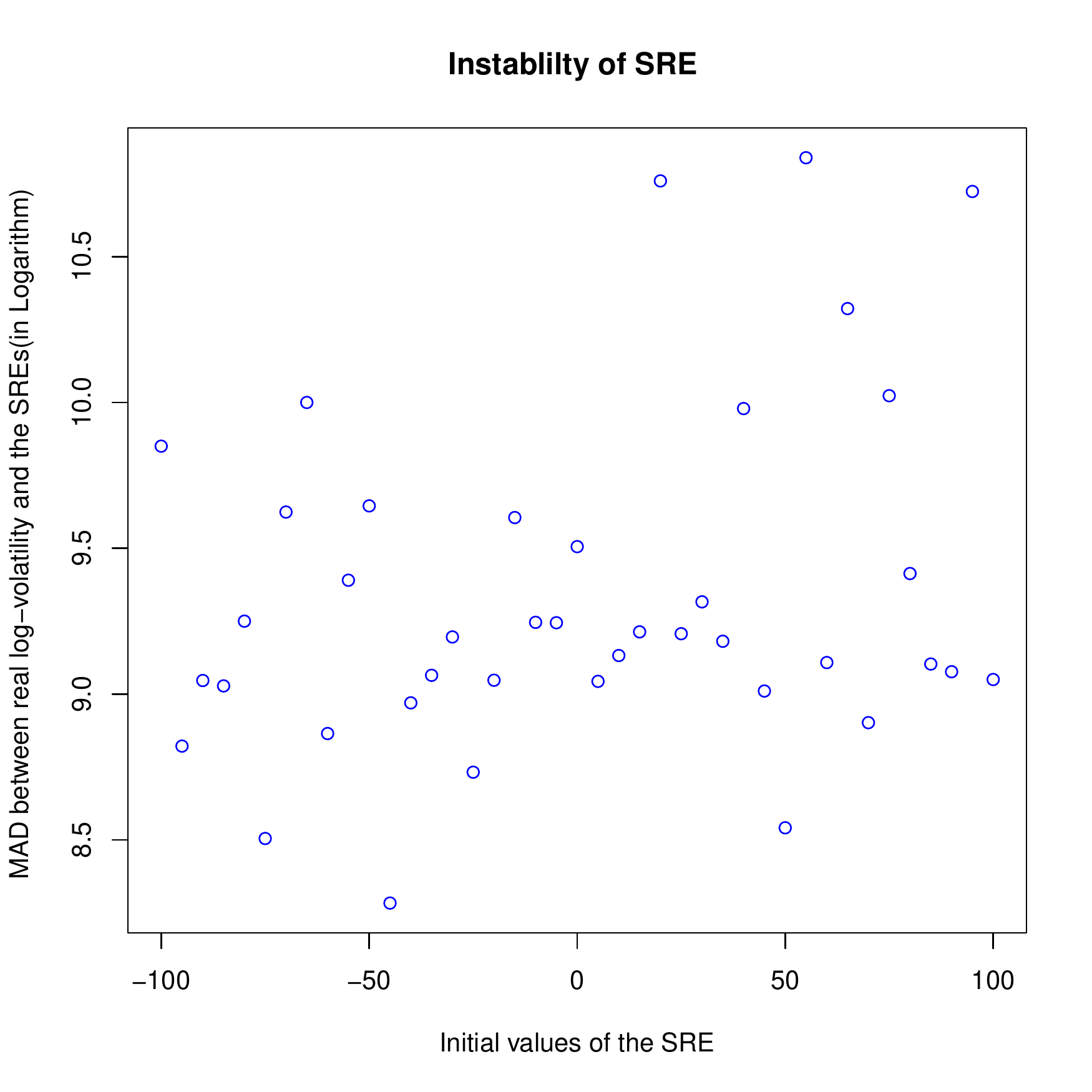}
\end{center}}
  \caption{Stable and non stable inverted EGARCH(1,1) models.}
 \label{fig:conv}
\end{figure}

\section{Strong consistency of the QMLE on continuously invertible domains}\label{sec:sc}
\subsection{Continuously invertible models}

Assume now that the DGP follows the general model \eqref{SREinftrue} for an unknown value $\theta_0$. Consider the inference of the QMLE 
$
\hat{\theta}_{n}$ defined as $\hat \theta_n =\mbox{argmin}_{\theta\in\Theta}\hat{L}_{n}(\theta)$ for some compact set $\Theta$. Here $\hat L_n$ is the QL criteria defined in \eqref{eq:ql} and based on  the approximation $\hat g_t(\theta)$ of the volatility:
$$
\hat{g}_{t+1}(\theta ) = \phi_t(\hat{g}_{t}(\theta ),\theta),\qquad\forall t\ge 0,\forall \theta\in\Theta
$$
starting at an arbitrary initial value $\hat{g}_{0}(\theta )$, for any $\theta\in \Theta$. The invertibility of the model is not sufficient to assert the consistency of the inference: as the parameter $\theta_0$ is unknown,  the conditions of invertibility \eqref{eq:inv} do not provide the stability of the approximation $\hat g_t(\theta)$ wrt the initial value $\hat{g}_{0}(\theta )$ when $\theta\neq \theta_0$. satisfied for $(\phi_t(\cdot,\theta))$ for any $\theta\in\Theta$ then the unique stationary solution $(g_t)$ exists and is defined by  the relation
\begin{equation}\label{sreg}
{g}_{t+1}(\theta ) = \phi_t({g}_{t}(\theta ),\theta),\qquad\forall t\in \Z, \, \forall \theta \in \Theta.
\end{equation}
Assume there exists some subset  $\mathcal K$ of $\R $ such that the random functions $(x,\theta)\to \phi_t(x,\theta)$  restricted on $\mathcal K\times \Theta$  take values in $\mathcal K$ for. Consider that the initial values $\theta\to\hat g_0(\theta)$ constitutes a continuous function on $\Theta$ that takes its values in $\mathcal K$. Denote $\|\cdot\|_\Theta$ the uniform norm and $\mathcal C_\Theta$ the space of continuous functions from $\Theta$ to $\mathcal K$.  By a recursive argument, it is obvious that the random functions $\hat g_t$ belong to $\mathcal C_\Theta$ for all $t\ge 0$.\\

We are now ready to define the notion of continuous invertibility:
\begin{Def}
The model is continuously invertible on $\Theta$ iff $\|\hat g_t(\theta)-g_{t}(\theta)\|_\Theta\to 0$ a.s. when $t\to \infty$.
\end{Def}
This continuous invertibility notion is   crucial   for the strong consistency of the QMLE. 

\subsection{Strong consistency of the QMLE constrained to a continuously invertible domain}

Let us assume the regularity and boundedness from below on the link function $\ell$:
\begin{description}
\item[(LB)]   The maps $x\to 1/\ell(x)$ and $\log(\ell(x))$ are Lipschitz functions on $\mathcal K$ and there exists $m>0$ such that $ \ell(x)\ge m$ for all $x\in\mathcal K$.
\end{description}
Classically, we also assume the identifiability of the model
\begin{description}
\item[(ID)]   
$\ell (g_0(\theta))=\sigma_0^2$ a.s. for some $\theta\in\Theta$
 iff $\theta=\theta_0$.
 \end{description}
The consistency of the QMLE follows from the continuous invertibility notion:
\begin{Theo}\label{th:sc}
Assume that {(\bf ST)}, {\bf (LB)} and {\bf (ID)} for a volatility model that is continuously invertible on the compact set $\Theta$. Then the QMLE on $\Theta$ is strongly consistent, $\hat\theta_n\to \theta_0$ a.s., when $\theta_0\in {\Theta}$.
\end{Theo}
\begin{proof}
First, note that from the continuous invertibility of the model on $\Theta$, the SRE  \eqref{sreg} admits a stationary solution in  $E=\mathcal C_\Theta$ that is a separable complete metric space (equipped with the metric $d(x,y)=\|x-y\|_\Theta$). Let us denote $\Phi_t$ the mapping acting on $\mathcal C_\Theta$ and satisfying
$$
g_{t+1}=\Phi_t(g_t)\quad\mbox{iff}\quad g_{t+1}(\theta)=\phi_t(g_t(\theta),\theta),\quad \forall \theta\in \Theta.
$$
We will apply the principle of Letac \cite{letac:1986} extended to the stationary sequences as in \cite{straumann:mikosch:2006}: the existence of a unique stationary non anticipative solution follows from  the convergence a.s. of the backward equation for any initial value $g\in \mathcal C_\Theta$:
$$
Z_t(g):=\Phi_0 \, o \,\Phi_{-1}  o\cdots o\, \Phi_{-t} (g)\to_{t\to \infty}Z
$$
where $Z$ does not depend on $g$. By stationarity, $Z_t(g)$ is distributed as $\hat g_t$ when $\hat g_0=g$. Thus, we have 
$$\P(\lim_{t\to\infty}\|Z_t(g)-g_0\|_\Theta=0)=\P(\lim_{t\to\infty}\|\hat g_t-g_t\|_\Theta=0)$$ where $g_0(\theta)$ is well defined for each $\theta\in\Theta$ as the solution of the SRE \eqref{sreg}. Thus, under continuous invertibility, an application of Letac's principle leads to the existence of a unique stationary solution distributed denoted also $(g_t)$ that coincides with $g_t(\theta)$ at any point $\theta\in\Theta$. In particular, $g_t$ belongs to $\mathcal C_\Theta$ and 
$$
2 n {L}_{n}(\theta)  =  \sum_{t=1}^{n} {l}_{t}(\theta) 
  =  \sum_{t=1}^{n} X_{t}^{2}/\ell( {g}_{t}(\theta))^{2}+\log(\ell( {g}_{t}(\theta)).
$$is a continuous function on $\Theta$.\\

Let us turn to the proof of the strong consistency based on standard arguments, see for example the book of Francq and Zako\" ian \cite{francq:zakoian:2010}. As the model satisfies the identifiability condition {\bf (ID)}, the strong consistency follows  the intermediate results
\begin{enumerate}
\item[(a)] $\lim_{n\to \infty} \sup_{\theta\in\Theta}|
L_n(\theta)-\hat{L}_n(\theta)|=0$  a.s.  
\item[(b)] $\E|l_0(\theta_0)|<\infty$ and if $\theta\neq\theta_0$ then $\E[l_0(\theta)]>\E[l_0(\theta_0)]$.
\item[(c)] Any $\theta\neq \theta_0$ has a neighborhood
$ V(\theta)$ such that 
$$
\lim\inf_{n\to\infty}\inf_{\theta^*\in
V(\theta)}\hat{L}_n(\theta^*)>\E l_t(\theta_0)\;\; \mbox{ a.s. }
$$
\end{enumerate}
 
Let us prove that (a) is satisfied when the model is continuously invertible and the  condition {\bf (LB)} holds. As $1/\ell$ and $\log( \ell)$ are Lipschitz continuous functions there exists some constant $C>0$ such that
\begin{equation}\label{eq:approx}
|
L_n(\theta)-\hat{L}_n(\theta)|\le C\frac1n\sum_{t=1}^n|\hat g_t(\theta)- g_t(\theta)|.
\end{equation}
The desired convergence to $0$ of the upper bound that is a Cesaro mean follows from the definition of the continuous invertibility.\\

The first assertion of (b) follows from the identity 
$$
 l_0(\theta_0)=\log(\sigma_0^2) + Z_0^2.
$$
Thus, as $\sigma_0=\ell(g_0(\theta_0))\ge m$ from {\bf (LB)} and $\E\log^+\sigma_0^2 <\infty$ under {\bf (ST)}, we have that $\log(\sigma_0^2)$ is integrable. Moreover, $\E Z_0^2=1$ by assumption  and the assertion $\E|l_0(\theta_0)|<\infty$  is proved. To prove the second assertion,  note that $\theta\to \E[l_0(\theta)]$ has a unique minimum iff $\E [\sigma_0^2/\ell(g_0(\theta))-\log(\sigma_0^2/\ell(g_0(\theta) )] $ has a unique minimum. Under the identifiability condition, as $x-\log(x)\ge 1$ for all $x>0$ with equality iff $x=1$, we deduce that for any $\theta\neq\theta_0$ we have $  \E[l_0(\theta)]> \E[l_0(\theta_0)]$.\\

Finally, let us prove (c) under continuous invertibility, {\bf (LB)} and {\bf (ID)}. First, under continuous invertibility, we have
\begin{eqnarray*}
\lim \inf_{\theta^*\to\theta}\hat L_n(\theta^*)&=&\lim\inf_{n\to\infty}\inf_{\theta^*\in
V(\theta)}{L}_n(\theta^*)+\lim\inf_{n\to\infty}\inf_{\theta^*\in
V(\theta)}(\hat{L}_n(\theta^*)-L_n(\theta^*))\\
&=&\lim\inf_{n\to\infty}\inf_{\theta^*\in
V(\theta)}{L}_n(\theta^*)
\end{eqnarray*}
by using the convergence to $0$ of the Cesaro mean \eqref{eq:approx}. Second, by ergodicity of the stationary solution $(g_t)$ we have the ergodicity of the sequence $(\inf_{\theta^*\in
V(\theta)}l_t(\theta^*))$. Moreover, under the condition {\bf (LB)}, we have that $\inf_{\theta^*\in
V(\theta)}l_t(\theta^*)\ge \log(m)>-\infty$ a.s. Then, for any $K>0$, the sequence $\inf_{\theta^*\in
V(\theta)}l_t(\theta^*)\wedge K$ is integrable. We use the classical SLLN and obtain
$$
\lim_{n\to\infty}\inf_{\theta^*\in
V(\theta)}L_n(\theta^*)\wedge K=\E\Big[\inf_{\theta^*\in
V(\theta)}l_0(\theta^*)\wedge K\Big]\qquad a.s.
$$
Letting $K\to - \infty$ we obtain that 
$$
\lim_{n\to\infty}\inf_{\theta^*\in
V(\theta)}L_n(\theta^*)=\E\Big[\inf_{\theta^*\in
V(\theta)}l_0(\theta^*)\Big]\in\R\cup\{+\infty\}.
$$
Finally, remark that by continuity of $\theta\to g_t(\theta)$ the function  $l_0$ is continuous. Thus,  for any $\varepsilon>0$ we can find a neighborhood $V(\theta)$  such that
$$
\E\Big[\inf_{\theta^*\in
V(\theta)}l_0(\theta^*)\Big]\ge \E[l_0(\theta)]+\varepsilon.
$$
From the second assertion of (b) we choose $\varepsilon>0$ such that $\E[l_0(\theta)]+\varepsilon>\E[l_0(\theta_0)]$ and (c) is proved. The end of the proof of the strong consistency is based on a classical compact argument and thus is omitted.
\end{proof}

\subsection{Sufficient conditions for continuous invertibility}
The definition of continuous invertibility does not give any explicit condition. In order to apply Theorem 3.1 in \cite{bougerol:1993}, one classically assumes some uniform Lipschitz conditions on the SRE \eqref{sreg}. Hereafter we will use a continuity argument instead to obtain a tractable sufficient condition of continuous invertibility for general models.\\

Consider some generic function $f:\mathcal K\times \Theta\mapsto \mathcal K$. Assume that there exists a continuous function $\Lambda_{f}$ on $\Theta$ such that for each $x$, $y\in E$ we have
$$
|f(x,\theta)-f(y,\theta)|\le \Lambda_{f}(\theta)|x-y|.
$$
The approach followed by Straumann and Mikosch in \cite{straumann:mikosch:2006} is to consider the SRE \eqref{SREgen} in the complete metric space of continuous functions $\mathcal C_\Theta$ on $\Theta$ with values in $\mathcal K$ equipped with $d(x,y)=\|x-y\|_\Theta$. Straightforward conditions for continuous invertibility are the following ones
\begin{multline}\label{eq:ui}
\E\log^+(\|\phi_0(y,\cdot)\|_\Theta)<\infty \mbox{ for some $y\in \mathcal K$, }\E\log^+(\|\Lambda_{\phi_0} \|_\Theta)<\infty \mbox{ and }  \E\log\big(\big\|\Lambda_{\phi_0^{(r)}} \big\|_\Theta\big)<0.
\end{multline}
The EGARCH(1,1) model satisfies this condition under restrictive assumptions on $\Theta$, for example when $\beta=0$ for any $\theta\in\Theta$, see \cite{straumann:2005}.\\

We collect more general sufficient  conditions for continuous invertibility in the assumption {\bf (CI)}
\begin{description}
\item[(CI)]  $\E\log^+(\|\phi_0(y,\cdot)\|_\Theta)<\infty$ for some $y\in \mathcal K$, $\E\log^+(\|\Lambda_{\phi_0} \|_\Theta)<\infty$ and  $\E\log\big(\Lambda_{\phi_0^{(r)}}(\theta)\big)<0$ for any $\theta\in\Theta$.
\end{description}
The difference with conditions \eqref{eq:ui} is that  the  Lyapunov  condition holds pointwisely and not necessarily uniformly on $\Theta$. Of course {\bf (CI)} is weaker than the uniform Lyapunov condition \eqref{eq:ui}. Due to the regularity of classical models wrt to $\theta$, {\bf (CI)} is satisfied as soon as the model is invertible on $\Theta$ (see Section \ref{sec:sc}  for the EGARCH(1,1) case). Next Theorem proves that the continuity argument is sufficient to assert continuous invertibility:
\begin{Theo}\label{th:cont}
If {\bf (ST)} and {\bf (CI)} hold on some compact set $\Theta$ then the model is continuously invertible on $\Theta$.
\end{Theo}

 \begin{proof} For any $\rho>0$, let us write
 $\Lambda_\ast^{(r)}(\theta,\rho)=\sup\{\Lambda_{\phi_0^{(r)}}(\theta^*),\theta^*\in \overline B(\theta,\rho)\cap \Theta\}$, where $\overline B(\theta,\rho)$ stands 
for the closed ball centered at $\theta$ with radius $\rho$. Note that for any $K>0$ we have  $\E[\sup_\Theta|\log \Lambda_{\phi_0}^{(r)}(\theta)\vee K|]<\infty$ because $\E[\sup_\Theta\log^+ \Lambda_{\phi_0^{(r)}}(\theta)]<\infty$ under {\bf (CI)}. Applying the dominated convergence theorem we obtain
$\lim_{\rho\to0}\E[\log \Lambda_{\ast}^{(r)}(\theta,\rho)\vee  K ]=\E[\lim_{\rho\to0}\log \Lambda_{\ast}^{(r)}(\theta,\rho)\vee K]$.
By continuity we have $\lim_{\rho\to0}\log \Lambda_{\ast}^{(r)}(\theta,\rho)= \log \Lambda_{\phi_0^{(r)}}(\theta)$ and for sufficiently small $K<0$
$$
\lim_{\rho\to0}\E[\log \Lambda_{\ast}^{(r)}(\theta,\rho)\vee K]= \E[\log \Lambda_{\phi_0^{(r)}}(\theta)\vee K]<0.
$$
Thus, there exists an $\epsilon>0$ such that 
$$\E[\log \Lambda_{\ast}^{(r)}(\theta,\epsilon)]\le\E[\log \Lambda_{\ast}^{(r)}(\theta,\epsilon)\vee K]<0.$$\

Denote $\mathcal V(\theta)$ the compact neighborhood $\overline B(\theta,\epsilon)\cap \Theta$ of $\theta$. Let us now work on $\mathcal C_{\mathcal V(\theta)}$, the complete metric space of continuous functions from $\mathcal V(\theta)$ to $\mathcal K$ equipped with the supremum norm $\|\cdot\|_{\mathcal V(\theta)}$. In this setting $(\hat g_t)$ satisfies a functional SRE $\hat g_{t+1} = \Phi_t(\hat g_t)$ with Lipschitz coefficients satisfying
\begin{align*}
\Lambda (\Phi_t^{(r)})&\le \sup_{s_1,s_2\in \mathcal C( \mathcal V(\theta))}
\frac{\|\Phi_t^{(r)}(s_1)-\Phi_t^{(r)}(s_2)\|_{\mathcal V(\theta)}}{\|s_1-s_2\|_{\mathcal V(\theta)}}\\
&\le 
\sup_{s_1,s_2\in \mathcal C(\mathcal V(\theta))}\frac{\sup_{\theta^*\in \mathcal V(\theta)}\|\phi_t^{(r)}(s_1(\theta^*),\theta^*)-\phi_t^{(r)}(s_2(\theta^*),\theta^*)\|}{\|s_1-s_2\|_{\mathcal V(\theta)}}\\
 &\le \sup_{s_1,s_2\in \mathcal C(\mathcal V(\theta))}\frac{\sup_{\theta^*\in \mathcal V(\theta)}\Lambda(\phi_t^{(r)}(\cdot,\theta^*))\|s_1(\theta^*)-s_2(\theta^*)\|}{\|s_1-s_2\|_{\mathcal V(\theta)}}\\
 &\le  \sup_{s_1,s_2\in \mathcal C(\mathcal V(\theta))}\frac{\sup_{\theta^*\in \mathcal V(\theta)}\Lambda(\phi_t^{(r)}(\cdot,\theta^*))\|s_1-s_2\|_{\mathcal V(\theta)}}{\|s_1-s_2\|_{\mathcal V(\theta)}}\le  \Lambda_{\ast}^{(r)}(\theta,\epsilon).
\end{align*}
As $\E[\log^+(\|\Phi_0(y)\|_{\mathcal V(\theta)})]\le \E[ \log^+(\|\phi_0(y,\theta)\|_\Theta)] <\infty$, $\E[\log^+( \Lambda(\Phi_0) )]\le \E[ \log^+(\|\Lambda_{\phi_0}(\theta)\|_\Theta)]<\infty$ under {\bf (CI)}   and $ \E(\log \Lambda(\Phi_0^{(r)}))<0$ we can apply Theorem 3.1 of \cite{bougerol:1993}. The unique stationary solution $(g_t)$ exists and satisfies 
\begin{equation}\label{eas}
\|\hat{g}_{t}-g_{t}\|_{\mathcal V(\theta)}\xrightarrow{ {e.a.s.}}0.
\end{equation}
Now, let us remark that  $\Theta=\cup_{\theta\in\Theta}\mathcal V(\theta)$.  As $\Theta$ is a compact set,  there exists a finite number $N$   such that $\Theta=\cup_{k=1}^N \mathcal V(\theta_k)$ and we obtain 
$$
\|\hat{g}_{t}-g_{t}\|_{\mathcal V(\theta_k)}\xrightarrow{{e.a.s.}}0\quad \forall 1\le k\le N.
$$
The desired result follows from the indentity $\|\cdot\|_\Theta=\vee_{k=1}^N\|\cdot\|_{\mathcal V(\theta_k)}.$
\end{proof}

\subsection{Volatility forecasting}

Based on the QMLE $\hat \theta_n$ we deduce a natural forecasting $\hat\sigma_{n+1}^2=\ell(\hat g_n(\hat\theta_n))$ of the volatility $\sigma_{n+1}^2$. It is strongly consistent:

\begin{Theo}\label{pr:vf}
Assume that {\bf (ST)}, {\bf (LB)}, {\bf (ID)} and {\bf (CI)} hold for $\theta_0\in\Theta$. Then  $|\hat\sigma_{n+1}^2- \sigma_{n+1}^2|\xrightarrow{ {a.s.}} 0$  as $n\to\infty$.
\end{Theo}
\begin{proof}
Theorems \ref{th:sc} and  \ref{th:cont}  assert the strong consistency of
 $\hat\theta_n$ under the assumption of Theorem \ref{pr:vf}. By continuity of $\ell$ and continuous invertibility on $\Theta$, $|\ell(\hat g_t(\hat\theta_n))-\ell( g_t(\hat\theta_n))|\xrightarrow{ {a.s.}} 0$. Thus, using again the continuity of $\ell$, the result is proved if $|g_t(\hat\theta_n)-g_t(\theta_0)|\xrightarrow{ {a.s.}} 0$. Keeping the notation used in the proofs of  Theorems \ref{th:sc} and  \ref{th:cont}, we have $\hat \theta_n\in \mathcal V(\theta_0)$  for $n$ sufficiently large. The  Lyapunov condition 
$\E[ \log \Lambda (\Phi_0^{(r)})]<0$ is satisfied for the functional SRE $(\Phi_t)$ restricted on $\mathcal V(\theta_0)$. For any $\theta\in \mathcal V(\theta_0)$, $t\in\Z$ we have
$$
|g_{(t+1)r}(\theta)-g_{(t+1)r}( \theta_0)| \le  \Lambda (\Phi_t^{(r)})|g_{t}(\theta)-g_{t}( \theta_0)| +w_t(\theta) $$
where $w_t(\theta)=|\phi_t(g_{t}(\theta_0), \theta )-\phi_t(g_{t}(\theta_0), \theta_0)|$. Thus $|g_{t}(\theta)-g_{t}( \theta_0)|$ is bounded by a linear SRE.   Applying the Borel-Cantelli Lemma as in  \cite{berkes:horvath:kokoszka:2003,straumann:mikosch:2006},  the convergence of the series
 $$
\|g_{tr} -g_{tr}( \theta_0)\|_{\mathcal V(\theta_0)}\le \sum_{i=0}^\infty\Lambda (\Phi_{tr}^{(r)} )\cdots\Lambda (\Phi_{(t-i+1)r}^{(r)})\|w_{(t-i)r}\|_{\mathcal V(\theta_0)}<\infty
$$
follows from the fact that  $\Lambda (\Phi_{tr}^{(r)} )\cdots\Lambda (\Phi_{(t-i+1)r}^{(r)})\xrightarrow{e.a.s.}0$ when $i\to \infty$ and $\E\log^+\|w_{t-i}\|_{\mathcal V(\theta_0)}<\infty$
Thus  the difference $|g_{tr}(\theta) -g_{tr}( \theta_0)|$ is bounding by an a.s. normally convergent function on $\mathcal C_{\mathcal V(\theta_0)}$ that we denote $a(\theta)$. Moreover, as $\lim_{\theta\to\theta_0}w_t(\theta)=0$ for any $t\in\Z$ a.s., we also have $\lim_{\theta\to\theta_0}a_t(\theta)=0$ a.s. from the normal convergence. Finally, using the strong consistency of $\hat\theta_n$, we obtain that
\begin{equation}\label{nt}
|g_{t}(\hat\theta_n)-g_{t}( \theta_0)|\le a(\hat\theta_n)\xrightarrow{ {a.s.}}0\quad \forall n\ge t\quad \mbox{when}\quad t\to\infty.
\end{equation}
\end{proof}

\section{Strong consistency and SQMLE for the EGARCH($1,1$) model}\label{sec:iqmle}

\subsection{Strong consistency of the QMLE when INV($\theta$) is satisfied for any $\theta\in\Theta$}\label{sec:sc}
Recall that $(Z_t)$ is an iid sequence of r.v.  such that $\E (Z_0)=0$ and  $\E(Z_0^2)=1$. 
The EGARCH($1,1$) model introduced by \cite{nelson:1991} is an AR($1$) model for $\log \sigma_t^2$,
$$
X_{t}  =  \sigma_{t}Z_{t}\quad\mbox{with}\quad
\log\sigma_{t}^{2}  =\alpha_0+\beta_0\log\sigma_{t-1}^{2}+\gamma_0 Z_{t}+\delta_0 |Z_{t} |.
$$
The volatility process $( \sigma_t^2)$ exists, is stationary and ergodic as soon as $|\beta_0|<1$ with no other constraint on the coefficients. and {\bf (ST)} holds. However, it does not necessarily 
have finite moment of any order. The model is identifiable as soon as the distribution of $Z_0$ is not concentrated on two points, see \cite{straumann:mikosch:2006}. Let us assume in the rest of the paper these classical conditions satisfied such that assumptions {\bf (ST)}, {\bf (LB)} and {\bf (ID)} automatically hold.\\

The continuous invertibility of the stationary solution of the EGARCH($1,1$) model does not hold on any compact set $\Theta$. Recall that the inverted model is driven by the SRE \eqref{eq:egarchinv}.
Under the constraint $\delta\ge |\gamma|$, we  have that $\gamma X_{t}+\delta |X_{t} |\ge 0$ a.s.. By a straightforward monotonicity argument, we can consider the restriction of the SRE \eqref{eq:egarchinv} on the intervall $\mathcal K\times \Theta$ where $\mathcal K=[\alpha/(1-\beta),\infty)$ and $\delta\ge |\gamma|$ for any $\theta\in\Theta$. By regularity of $\phi_t$ wrt $x$, the Lipschitz coefficients are computed using the first partial derivative:
$$
\Lambda(\phi_{t}(\cdot,\theta ))\le \max \{\beta ,2^{-1}(\gamma  X_{t}+\delta |X_{t}|)\exp(-2^{-1}\alpha /(1-\beta ))-\beta\}.
$$
This Lipschitz coefficient is continuous in $\theta$ and thus coincides with $\Lambda_{\phi_t}(\theta)$. The assumption {\bf (CI)} is satisfied if INV($\theta$) is satisfied for any $\theta\in\Theta$ as the uniform log moments exists by continuity and because $\E[ \log^+ X_0] <\infty$. As {\bf (ST)} is also satisfied, an application of Theorem \ref{th:cont} asserts the continuous invertibility on any compact sets $\Theta$ such that INV($\theta$) is satisfied for any $\theta\in\Theta$. \\
 
Remark that the domain of invertibility represented in
Figure \ref{fig:Beta} does not coincide with the domain of continuous invertibility, i.e. the set of $\theta\in\R^4$ satisfying INV($\theta$) for an EGARCH(1,1) DGP with  $\theta_0\in \R^4$. This situation is more complicated to represent as the domain INV($\theta$) depends on the fixed value $\theta_0$ and   on the parameter $\alpha$ when $\theta\neq\theta_0$. However, for any $\theta_0$ in the invertibility domain represented in Figure \ref{fig:Beta}, there exists some compact neighborhood $\Theta$ of continuous invertibility. For the QMLE constrained to such compact set $\Theta$, an application of Theorems \ref{th:sc}, \ref{th:cont} and \ref{pr:vf} gives directly
\begin{Theo}\label{cor:sc}
Consider the EGARCH(1,1) model. If INV($\theta$) is satisfied for any $\theta\in\Theta$ and $\theta_0\in {\Theta}$ then $\hat\theta_n\to\theta_0$ and $\hat\sigma_n^2-\sigma_n^2\to 0$ a.s. as $n\to \infty$ with $\hat\sigma^2_n=\exp(\hat g_n(\hat \theta_n))$ for any initial value $\hat\sigma_0^2$.
\end{Theo}
Remark we extend the result of  \cite{straumann:2005,straumann:mikosch:2006} under the sufficient condition \eqref{eq:ui} that expresses in the EGARCH(1,1) model as
$$
\E[\sup_\Theta\log(\max\{ \beta , 2^{-1}(\gamma  X_{t-1}+\delta |X_{t-1}|)\exp(-2^{-1}\alpha /(1-\beta ))-\beta \})]<0.
$$
This more restrictive condition is less explicit than \eqref{eq:condtheta}  because it is a Lyapunov condition uniform on $\Theta$. It is difficult to check in practice when $\beta\neq0$ for some $\theta\in\Theta$.

\subsection{The Stable QMLE (SQMLE)}
Noe that the procedure is valid only if the invertibility condition INV($\theta$) is satisfies for any $\theta\in\Theta$. Thus, we want to constrain the optimization of the QL on a continuously invertible domain. However, the condition \eqref{eq:condtheta} depends on the distribution of $X_0$ and on the unknown parameter $\theta_0$ that drives the DGP. We propose to constrain the optimization of the QL under the empirical constraint\\

$\widehat{\mbox{INV}(\theta)}$: $\delta\ge |\gamma|$ and 
$
\sum_{t=1}^n\log(\max\{\beta,2^{-1}(\gamma X_{t}+\delta|X_t|)\exp(-2^{-1}\alpha/(1-\beta))-\beta\})\le -\varepsilon.
$\\

We introduce artificially $\varepsilon>0$ (as small as we want) such that $\widehat{\mbox{INV}(\theta)}$ defines a closed set on $\R^4$. This set is non empty because the constraint is satisfied for $\theta\in\R^4$ such that $\beta=0$ and the parameters $\delta\ge |\gamma|$ and $\alpha$ are sufficiently small.
\begin{Def}
The SQMLE is the $M$-estimator
$$
\hat\theta_{n}^S=\mbox{argmin}_{\theta\in \Theta^S} \sum_{t=1}^{n}2^{-1}\left(X_{t}^2\exp(-\hat g_t(\theta)/2)+ \hat{g}_{t}(\theta)\right)
$$
where $\Theta^S=\{\theta\in\Theta\mbox{ satisfying }\widehat{\mbox{INV}(\theta)}\}$ for any compact set $\Theta$.
\end{Def}
Consider in the sequel that $\Theta^S$ is non empty. It is always the case in practice where we use a steepest descent algorithm on the constraints $\widehat{\mbox{INV}(\theta)}$ and some maximum number of iterations.
The following theorem gives the strong consistency of the SQMLE for the  EGARCH(1,1) model if INV($\theta_0$) is satisfied and $\theta_0\in\Theta$. It also shows that the volatility forecasting 
using the SQMLE does not depend on the arbitrary choice of the initial value even when INV($\theta_0$) is not satisfied. Thus the  SQMLE is more reliable than the QMLE for which chaotic behavior of the volatility forecasting described in \cite{sorokin:2011} can occur if INV($\theta$) is not satisfied for some $\theta\in\Theta$.
\begin{Theo}
If INV($\theta_0 $) and $\theta_0\in\Theta$  then $\hat\theta_{n}^S\xrightarrow{a.s.}\theta_0$. If INV($\theta_0$) is not satisfied, the asymptotic law of $\hat \sigma^2_t=\exp(\hat g_t(\hat\theta_{n}^S))$ still does not depend on the initial value $\hat \sigma^2_0$.
\end{Theo}
\begin{proof}
Denote $ \Lambda_n(\theta)=\max\{ \beta , 2^{-1}(\gamma  X_{n}+\delta |X_{n}|)\exp(-2^{-1}\alpha /(1-\beta ))-\beta \}$. We will prove that
$
 \|n^{-1} \log \Lambda_n(\theta) - \E\log\Lambda_0(\theta)  \|_{\Theta}\xrightarrow{  {a.s.}} 0$ as $n\to \infty$. Then, if INV($\theta_0$) is satisfied and $\theta_0\in\Theta$, $\Theta^S$ coincides  asymptotically a.s. to a compact (because bounded and close) continuously invertible domain containing $\theta_0$. An application of Theorem \ref{cor:sc} yields the strong consistency of the SQMLE. The second assertion follows from the fact that asymptotically INV($\hat\theta_{n}^S$) is satisfied even if $\theta_0\notin\Theta^S$.\\

Let us prove that $
 \|n^{-1} \log \Lambda_n(\theta) - \E\log\Lambda_0(\theta)  \|_{\Theta}\xrightarrow{  {a.s.}} 0$. Note that $\log \Lambda_n(\theta)$ is a random element in the Banach space $\mathcal C(\Theta)$.  The desired result is a consequence of the ergodic theorem if 
$
\E\|\log\Lambda_0\|_\Theta<\infty.
$
Since $$\|\log\Lambda_0\|_\Theta\le \max\{\|\log |\beta|\|_\Theta, \|\log |2^{-1}(\gamma  X_{0}+\delta |X_{0}|)\exp(-2^{-1}\alpha /(1-\beta ))-\beta|\|_\Theta\}$$  and $\Theta$ is a compact set, the desired result follows by continuity and the dominated convergence theorem as $\E\log|X_0|=\E\log\sigma_0+\E\log|Z_0|<\infty$. 
\end{proof}
The advantage of the SQMLE  $\hat\theta_n^{S}$ is that the procedure is stable wrt the choice of the initial value, whereas the QMLE is not stable if there is one $\theta\in\Theta$ satisfiying the non invertibility condition of \cite{sorokin:2011}. The stabilization procedure relies on the explicit continuous invertibility condition {\bf (CI)}. The expression of the sufficient constraint $\widehat{\mbox{INV}(\theta)}$ is specific to the EGARCH(1,1) model. We believe that this constraint is not sharp; it should be possible to improve the invertibility condition INV($\theta$) of \cite{straumann:2005}, thus to improve {\bf (CI)}, to extend the constraint $\widehat{\mbox{INV}(\theta)}$ and finally to obtain more general SQMLE. Remark that such stabilization of the QMLE should be done before using this classical estimator on models that are not everywhere invertible nor explosive. 

\section{Asymptotic normality of the SQMLE in the EGARCH(1,1) model}\label{sec:an}

In this section we extend the result of Theorem 5.7.9 of \cite{straumann:2005}  to non-degenerate cases when $\beta_0\neq0$.  We obtain the asymptotic normality of the SQMLE without assuming  uniform moment on the compact set $\Theta$ for the likelihood and its derivatives. However, we assume the additional condition 
\begin{description}
\item[(MX)] The EGARCH(1,1) volatility DGP $(\sigma_t^2)$ is geometrically ergodic.
\end{description}
The geometric ergodicity is a classical assumption in the context of Markov chains, see \cite{meyn:tweedie:1993}. Condition {\bf (MX)} is satisfied if $Z_0$ is nonsingular w.r.t. the Lebesgue measure on $\R$, see \cite{alsmeyer:2003}. Geometric ergodicity is equivalent to the strong mixing property of the DGP with a geometric rate of decrease of the coefficients. The notion of geometric strongly mixing is not restricted to the case of Markov chains. As we will use it in the proof, let us recall that it is equivalent to 
the existence of $0<a<1$ and $b>0$ such that
$\alpha_r\le ba^r$ for all $r\ge 1$. Here the $\alpha_r$, $r\ge1$,  are the strong mixing coefficients defined by \cite{rosenblatt:1956} as 
$$\alpha_r=\sup_{A\in \sigma(\ldots,X_{-1},X_0)\,,B\in
  \sigma(X_r,X_{r+1},\ldots)}
\left|\P(A\cap B)-\P(A)\,\P(B)\right|. 
$$

Let us also assume the finite moments assumption which is necessary and sufficient for the existence of the asymptotic covariance matrix, see Lemma \ref{lem:mom} below:
\begin{description}
\item[(MM)]  $\E[Z_0^4]<\infty$ and $\E[(\beta_0 -2^{-1}(\gamma_0 Z_0+\delta_0\left|Z_0\right|))^2]<1$.
\end{description}

\begin{Theo}\label{anegarch}
Assume that the conditions of Theorem \ref{cor:sc} are satisfied, that $\theta_0\in \stackrel{\circ}{\Theta}$ and that {\bf (MX)} and {\bf (MM)} hold. Then  $\sqrt n(\hat\theta_n^{S}-\theta_0)\xrightarrow{d}\mathcal N(0,{\bf \Sigma})$ where ${\bf \Sigma}$ is an invertible matrix.
\end{Theo}

\begin{proof}
As the SQMLE is asymptotically a.s. equivalent to the QMLE under the assumptions of Theorem \ref{cor:sc} we will prove the result only for this last estimator. We first prove that Assumption {\bf (MM)} yields the existence of the asymptotic variance  
$${\bf \Sigma}={\bf J}^{-1}{\bf I} {\bf J}^{-1}$$
with ${\bf J}=\E[\mathbb H \ell_0(\theta_0)]$
and ${\bf I}=\E[\nabla \ell_0(\theta_0)\nabla \ell_0(\theta_0)^T]$, where $\mathbb H \ell_0(\theta_0)$ and $\nabla \ell_0(\theta_0)$ are the Hessian matrix and the gradient vector of $\ell_0$ at the point $\theta_0\in \stackrel{\circ}{\Theta}$. 

\begin{lem}\label{lem:mom}
Under condition {\bf(MM)} the covariance matrix  ${\bf \Sigma}$ exist and it is  an invertible matrix.
\end{lem}
The proof of this lemma is given in Section \ref{sec:pr}. Then we prove that the functions $\hat g_t$ and $g_t$ are twice continuously differentiable refining the arguments developed in \cite{straumann:mikosch:2006}
\begin{lem}\label{lem:der}
Under the assumptions of Theorem \ref{cor:sc} the functions $\hat g_t$ and $g_t$ are twice continuously differentiable on a compact neighborhood $\mathcal V( \theta_0)$ of $\theta_0\in\stackrel{\circ}{\Theta}$ and $\|\nabla \hat g_t-\nabla  g_t\|_{\mathcal V( \theta_0)}+\|\mathbb H\hat g_t-\mathbb H  g_t\|_{\mathcal V( \theta_0)}\xrightarrow{e.a.s.}0$.
\end{lem}
The proof of this lemma is given in Section \ref{sec:pr}.  The asymptotic normality follows from the Taylor development used in Section 5 of Bardet and Wintenberger \cite{bardet:wintenberger:2009} on the partial derivatives $\nabla_i$ of the real valued function $L_n$:
$$
\nabla_iL_n(\hat\theta_n)-\nabla_iL_n( \theta_0)=\mathbb H L_n(\tilde\theta_{n,i})(\hat\theta_n-\theta_0)\quad\mbox{for all}~1\le i\le d.
$$
Then the asymptotic normality follows from the following sufficient conditions:
\begin{enumerate}
\item[(a)]  $n^{-1/2}\nabla L_n(\theta_0)\to \mathcal N(0,{\bf I})$,
\item[(b)]  $\|n^{-1}\mathbb H L_n(\tilde\theta_n)-{\bf J}\|\xrightarrow{a.s.}0$ for any sequence $(\tilde \theta_n)$ converging a.s. to $\theta_0$ and  ${ \bf J}$  is invertible,
\item[(c)]  $ n^{-1/2}\|\nabla \hat L_n(\hat\theta_n)-\nabla L_n(\hat\theta_n)\|$
  converges a.s. to $0$.
\end{enumerate}
Note that 
$$\nabla L_n(\theta_0)=\sum_{i=t}^n4^{-1} \nabla g_{t}(\theta_{0})(1-\Z_{t}^{2}) $$
is a martingale as $\nabla g_{t}(\theta_{0})$ is independent of $Z_t$ for all $t\ge 1$ and $\E[1-\Z_{t}^{2})]=0$ by assumption. Under {\bf (MM)}, this martingale has finite moments of order 2 due to Lemma \ref{lem:mom}. An application of the CLT for differences of martingales of \cite{billinsgley:1999} yields (a).\\

The following Lemma is used to prove (b) without uniform moment assumption   on the compact set $\Theta$ for the likelihood and its derivatives
\begin{lem}\label{lem:appr}
Under the assumptions of Theorem \ref{anegarch} we have $\|\mathbb H L_n(\tilde\theta_n)-\mathbb H L_n(\theta_0)\|\xrightarrow{a.s.} 0$ for any sequence $(\tilde \theta_n)$ converging a.s. to $\theta_0$ when $n\to\infty$.
\end{lem}
The proof of this lemma is given in Section \ref{sec:pr}. Applying Lemma \ref{lem:appr}, it is sufficient to prove that  $\|n^{-1}\mathbb H L_n(\theta_0)-{\bf J}\|\xrightarrow{a.s.}0$ to obtain that
$\|n^{-1}\mathbb H L_n(\tilde\theta_n)-{\bf J}\|\xrightarrow{a.s.}0$.
The ergodic Theorem applied to the process $(\mathbb H l_t(\theta_0))$ (integrable under {\bf (MM)}) yields $\|n^{-1}\mathbb H L_n(\tilde\theta_n)-{\bf J}\|\xrightarrow{a.s.} 0$. Thus the first assertion of (b) is proved. The fact that ${\bf J}$ is an invertible matrix is already known, see Lemma \ref{lem:mom}.\\

Finally (c) is obtained by using the exponential decrease of the approximation  \eqref{eas}  that holds uniformly on some compact neighborhood  $\mathcal V(\theta_0)$ of $\theta_0$ and the identity 
$$\nabla L_n =\sum_{i=t}^n4^{-1} \nabla g_{t} (1-X_{t}^{2}\exp(-g_t)). $$
Note that here we use the fact that $\exp(-x)$ is a Lipschitz function on $[c,\infty)$ where $c:=\min_{\mathcal V(\theta_0)} \alpha/(1-\beta)$.
\end{proof}

\section{Proofs of the technical lemmas}\label{sec:pr}

\subsection{Proof of Lemma \ref{lem:mom}}

Let us denote $U_{t}=(1,\log\sigma_{t}^{2},Z_{t},|Z_{t}|)$ and $V_{t}=\beta_0 -2^{-1}(\gamma_0 Z_{t}+\delta_0\left|Z_{t}\right|)$.
 Then $(\nabla g_{t}(\theta_{0}))$ is the solution of the linear SRE
$$
\nabla g_{t+1}(\theta_{0})=U_{t}+V_{t}\nabla g_{t}(\theta_{0}),\qquad \forall t\in\Z.$$
Let us consider the process $Y_t=(\nabla g_{t}(\theta_{0}),\log(\sigma_t^2))'\in\R^5$. It satisfies the relation
$$
Y_{t+1}=\left(\begin{matrix}V_t&0&0&0&0\\
0&V_t&0&0&1\\
0&0&V_t&0&0\\
0&0&0&V_t&0\\
0&0&0&0&\beta_0 \end{matrix}\right)Y_t+\left(\begin{matrix}1\\0\\Z_t\\|Z_t|\\\alpha_0+\gamma_0Z_t+\delta_0|Z_t|\end{matrix}\right)=:\Gamma_tY_t+R_t,\qquad \forall t\in\Z.
$$
Thus $(Y_t)$ is a random coefficients autoregressive processsatisfying the assumptions of Theorem 4 (a) of \cite{tweedie:1998} iff $\E V_t^2<1$. By a direct application of the Theorem 4 (a) of \cite{tweedie:1998} we obtain that the process $(Y_t)$ is second order stationary and thus the existence of the matrix $\bfB=\E [\nabla g_{t}(\theta_{0})(\nabla g_{t}(\theta_{0}))^{T} ]$.\\

Let us prove that $\bfB$ is invertible.  By classical arguments, it is sufficient to prove that the components of the vector $\nabla g_0(\theta_0)$ are linearly independent. It is the case in the AGARCH(1,1) model as soon as the density of $Z_0$ is not concentrated on two points, se Lemma 8.2 of \cite{straumann:mikosch:2006}. Thus $\bfB$ is an invertible matrix.\\

Finally,  we have the identity ${\bf I}=2^{-1}\bfB$ as
$$
{\bf I}  =2^{-1}\E\left[(\nabla g_{t}(\theta_{0})(\nabla g_{t}(\theta_{0}))^{T}Z_{0}^{2}+\mathbb Hg_{t}(\theta_{0})(1-Z_{0}^{2})\right]
   =2^{-1}\E [\nabla g_{t}(\theta_{0})(\nabla g_{t}(\theta_{0}))^{T} ]=2^{-1}\bfB.
$$
We also have the identity 
${\bf J}=4^{-1}(\E Z_0^4-1)\bfB$ because 
\begin{multline*}
{\bf J}  =\E\left[4^{-1}\E\left[\nabla g_{t}(\theta_{0})(\nabla g_{t}(\theta_{0}))^{T}(1-\Z_{t}^{2})^{2}\right]|\cF_{t-1}\right]\\
   =4^{-1}\E[(1-\Z_{0}^{2})^{2}]\E[\nabla g_{t}(\theta_{0})(\nabla g_{t}(\theta_{0}))^{T}]=4^{-1}(\E Z_{0}^{4}-1) \bfB.
\end{multline*}
Thus, using the identity ${\bf \Sigma}=(\E Z_0^4-1)\bfB^{-1}$, this matrix exists and is  invertible. The lemma is proved.

\subsection{Proof of Lemma \ref{lem:der}}
The proof of the existence of the derivatives of the process $(g_t)$ requires a refinement of Theorem 3.1 of \cite{bougerol:1993}.  We give this new result in full generality when the SRE is on a Polish space $(E,d)$. 
A map $f: E\to E$ is a Lipschitz map if
$\Lambda(f)=\sup_{(x,y)\in E^2}d(f(x),f(y))/d(x,y)$
is finite.     The regularity property we study is the following one: let $(U_t)_{t\ge 0}$ be a sequence of  non negative r.v.
\begin{description}
\item[(EAS)] For any non negative sequence $W_t\xrightarrow{e.a.s.}0$ the series $(W_tU_t)$ converges a.s.
\end{description}
Remark that {\bf (EAS)} is implied for stationary sequences by a condition of log-moment of order 1 as in \cite{bougerol:1993} (it is a straightforward application of the Borel Cantelli Lemma also used in \cite{berkes:horvath:kokoszka:2003,straumann:mikosch:2006}). More generally,   {\bf (EAS)} is satisfied for any non necessarily stationary sequence $(Y_t)$ such that the series $(\P(U_t\ge \rho^t))=(P(\log^+U_t\ge t\varepsilon))$ converge for any $0<\rho<1$ and $\varepsilon>0$. Remark also that {\bf (EAS)} is also automatically satisfied for any sequence $(U_t)$ such that $U_t\xrightarrow{{e.a.s.}}0$. The property {\bf (EAS)} is very useful in our context as it also satisfied for any solution of a convergent SRE under minimal additional assumptions: finite log-moments of order $p>2$ and the geometric strongly mixing condition
\begin{Theo} \label{th:bo}
Let $(\Psi_{t})$ be a stationary ergodic sequence of Lipschitz maps from $E$ to $E$ that is also strongly mixing with geometric rate. Assume that 
$ (d(\Psi_t(x),x)) $ satisfies {\bf (EAS)} for some $x\in E$,  $\E[(\log^{+}\Lambda(\Psi_{0}))^p]<\infty$ for some $p>2$ and 
\begin{equation}\label{tl}
\E[\log\Lambda(\Psi_{0}^{(r)})]=\E[\log\Lambda(\Psi_{0}\circ\cdots\circ\Psi_{-r+1})]<0 \mbox{  for some }r\ge 1.
\end{equation}
Then the SRE $
Y_{t+1}=\Psi_{t}(Y_{t})$,  $t\in\Z$, converges: it admits a unique stationary solution $(Y_{t})_{t\in\Z}$ which is
ergodic and for any $y\in E$
$$
Y_{t}=\lim_{m\rightarrow\infty}\Psi_{t}\circ\cdots\circ\Psi_{t-m}(y),\quad t\in\Z.
$$
The 
$Y_{t}$ are measurable with respect to the $\sigma(\Psi_{t-k}, k\geq0)$ and
$$
d(\hat Y_{t},Y_{t})\xrightarrow{{e.a.s.}}0,\quad t\rightarrow\infty
$$
for $(\tilde Y_t)$ satisfying $\tilde{Y}_{t+1}=\Psi_t(\hat Y_{t})$, $t\ge0$, and $\hat Y_0=y$ for any $y\in E$. Moreover $d(Y_t,y)$ satisfies {\bf (EAS)} for any $y\in E$. \end{Theo}
\begin{proof}  Note that the proof of the existence of the stationary solution is due to Elton \cite{elton:1990}. That the approximation scheme is e.a.s. convergent is due to Bougerol \cite{bougerol:1993}. Both results hold under the assumption that the Lipschitz coefficients have a finite log-moment of order 1. A careful look at the proof of  Theorem 3.1 of \cite{bougerol:1993} shows that the condition $\E[\log^+d(\Psi_0(x),x)]<\infty$ is only used there to assert {\bf (EAS)} on $ (d(\Psi_t(x),x)) $. Thus, the first assertions follow the classical arguments developed in \cite{elton:1990,bougerol:1993}. It remain to prove the {\bf (EAS)} property only.\\

First notice that $(d(\hat Y_t,Y_t))$ satisfies {\bf (EAS)} because $d(\hat Y_{t},Y_{t})\xrightarrow{{e.a.s.}}0$.  Let us show that $d(\hat Y_t,y)$ satisfies {\bf (EAS)} when $\hat Y_0=y$. Fix $K<0$ such that $\E[\log\Lambda(\Psi_0^{(r)})\vee K]\le \log a$ for some $0<a<1$. Then 
$$
d(\hat Y_t,y)\le   \sum_{j=1}^td(\hat Y_j,\hat Y_{j-1})\le   \sum_{j=1}^t\Lambda(\Psi_j^{(j-1)})d(\Psi_{1}(y),y)\le \frac{ \sup_{1\le j\le t}\Lambda(\Psi_j^{(j-1)})a^{-j/r} }{1-a^{1/r}}d(\Psi_{1}(y),y).
$$
Let us prove that $\sup_{1\le j\le t}\Lambda(\Psi_j^{(j-1)})a^{-j/r}$ satisfies {\bf (EAS)}. It is implied by the Borel-Cantelli Lemma from the convergence of the series
$$
\Big(\P\Big(\sup_{1\le j\le t}\log\Lambda(\Psi_j^{(j-1)})- j\log(a^{1/r})\ge  t\varepsilon\Big)\Big).
$$ 
By subadditivity, assuming that $t/r\in\N$ for convenience,  we have 
$$
\P\Big(\sup_{1\le j\le t}\log\Lambda(\Psi_j^{(j-1)})- j\log(a^{1/r})\ge  t\varepsilon\Big)\le\P\Big(\sup_{1\le j\le t/r}\sum_{j=1}^{t/r}\log\Lambda(\Psi_{jr}^{(r)})- \frac tr\log(a)\ge  t\varepsilon\Big).$$
This series converges by an application of Theorem 1 in \cite{shao:1993} as $\log(\Lambda(\Psi_t))$ has a finite moment of order $p>2$ and is strongly mixing with geometric rate. Finally, for any non negative sequence $W_t\xrightarrow{e.a.s.}0$  the series $(W_t \sup_{1\le j\le t}\Lambda(\Psi_j^{(j-1)})a^{-j/r})$ converges and the desired result follows.
\end{proof}

We are now ready to state the following refinement of the theorem 2.10 of \cite{straumann:mikosch:2006} used to prove that the functions $\hat g_t$ and $g_t$ are twice continuously differentiable in a neighborhood of $\theta_0$:
\begin{Theo}\label{th:pert}
Let $B$ be a separable Banach space and $(\Psi_t)$ be a stationary ergodic sequence of Lipschitz maps from $B$ into $B$ that is strongly mixing with geometric rate. Assume that 
\begin{description}
\item[S]   $( \|\Psi_t(0)\|)$ satisfies {\bf (EAS)},   $\E[(\log^+ \Lambda(\Psi_0))^p]<\infty$ for some $p>2$  and  $\E[\log\Lambda(\Psi_0^{(r)})]<0$ for some $r\ge 1$.
\end{description} 
Let $(\hat\Psi_t)_{t\in\N}$ be a sequence of Lipschitz maps such that 
\begin{description}
\item[S'] $\|\hat \Psi_t(0)-\Psi_t(0)\|\xrightarrow{e.a.s.}0$ and $\Lambda(\hat \Psi_t-\Psi_t) \xrightarrow{e.a.s.}0$ as $t\to\infty$.
\end{description}
Then the unique stationary solution $Y_t$ of the SRE $\,Y_{t+1}=  \Psi_t(  Y_t)$, $t\in\Z$, exists, $(\|Y_t\|)$ satisfies {\bf (EAS)} and for every solution $(\hat Y_t)_{t\in\N}$ of the perturbed SRE $\hat Y_{t+1}=\hat \Psi_t(\hat Y_t)$, $t\ge 0$, we have $\|\hat Y_t-Y_t\|\xrightarrow{e.a.s.}0$ regardless the initial value $\hat Y_0\in B$.
\end{Theo}
{\em Proof of Theorem \ref{th:pert}.}\,
We assume the same conditions than in the theorem 2.10 of \cite{straumann:mikosch:2006} except that the conditions $\E[\log^+\|\Psi_0(0)\|]<\infty$ and $\E[\log^+\|Y_0\|]<\infty$ do not hold. Remark that we introduce this new approach based on the {\bf (EAS)} property due to the difficulty to check this last condition on the derivatives of the SRE \eqref{eq:egarchinv}.  A careful look at the proof of Theorem 2.10  in  \cite{straumann:mikosch:2006} shows that these conditions are used for the convergence of series of the form $(W_t\|\Psi_t(0)\|)$ and $(W_t'\|Y_t\|)$  converge for some $(W_t)$ and $(W_t')$ such that $W_t,W_t'\xrightarrow{e.a.s.}0$. The first series converges by assumption, the second one converges as $(\|Y_t\|)$ satisfies {\bf (EAS)} from the last assertion of Theorem \ref{th:bo}.$\hfill\square$\\

Let us come back to the proof of Lemma \ref{lem:der}. First note that as $(\sigma_t^2)$ is strongly mixing with geometric rate it is also the case of the process 
$$
(\phi_t(x,\theta))=(\alpha+\beta x+(\gamma X_{t}+\delta |X_{t} |)\exp(-x/2)).
$$  
Deriving the SRE \eqref{eq:egarch}, we obtain the new linear SRE
\begin{equation}\label{sredg}
\nabla\hat g_{t+1} (\theta)= \phi_t'(\hat g_t(\theta),\theta)\nabla \hat g_t(\theta)+\nabla_{\theta} \phi_t(\hat g_t(\theta),\theta),\qquad  t\ge 0
\end{equation}
where 
\begin{align*}
\phi_t'(x,\theta)&=\beta-2^{-1}(\gamma X_t+\delta|X_t|)\exp(-x/2),\\\nabla_{\theta} \phi_t(x,\theta)&=(1,x,X_t\exp(-x/2),|X_t|\exp(-x/2))'.
\end{align*}
Note that both functions $\phi_t'(x,\theta)$ and $\nabla_{\theta} \phi_t(x,\theta)$ are continuous in $(x,\theta)$.
Under {\bf (CI)}, we derive from an application of Theorem \ref{th:cont} the existence of a compact neighborhood $\mathcal V(\theta_0)\subset \stackrel{\circ}{\Theta}$ of $\theta_0$ such that $\E[\log \|\Lambda(\phi_0^{(r)}) \|_{\mathcal V(\theta_0)}]<0$ for some $r\ge 1$. 
The SRE \eqref{sredg} is a linear perturbed SRE $\nabla\hat g_{t+1}=\hat\Phi_t'(\nabla\hat g_{t})$ on the Banach space of continuous functions on $\mathcal V(\theta_0)$. The sequence $(\hat\Phi_t')$ is not stationary but it is well approximated by the stationary sequences $(\Phi_t')$ defined by the relation
$$
\Phi_t'(h):=\phi_t'( g_t ,\cdot)h+\nabla_{\theta} \phi_t( g_t ,\cdot), \qquad t\in \Z.$$\

We will apply Theorem \ref{th:pert} on the SREs driven by $(\hat \Phi_t')$ and $(\Phi_t')$ to assert the existence of the first derivatives $(\nabla g_t)$.  Let us check the first condition of the assumption {\bf S} of  Theorem \ref{th:pert}. The series $(\|\Phi_t'(0)\|_{\mathcal V(\theta_0)})=(\| \nabla_{\theta} \phi_t( g_t ,\cdot)\|_{\mathcal V(\theta_0)})$ satisfies {\bf (EAS)} because $\E\log^+ |X_t|<\infty$ and  $(\|\hat g_t\|_{\mathcal V(\theta_0)})$ satisfies {\bf (EAS)} by an application of Theorem \ref{th:bo} on $(g_t)$.  By definition of $\mathcal V(\theta_0)$, as $ \phi_t'(x,\theta)\le \Lambda (\phi_0,\theta)$, we check the last two conditions of the assumption {\bf S} by the continuous invertibility of the model and because $\E(\log^+X_0)^p<\infty$ for $p=8$ as $\E Z_0^4<\infty$.\\

Let us check the first condition of the assumption {\bf S'} of the theorem \ref{th:pert}. As $\phi_t$ is twice continuously differentiable, we have
$$
\|\Phi_t'(0)-\hat\Phi_t'(0)\|=\|\nabla_{\theta} \phi_t( g_t ,\cdot)-\nabla_{\theta} \phi_t( \hat g_t ,\cdot)\|_{\mathcal V(\theta_0)}\le \Big\|\sup_{x\ge c }\nabla_\theta \phi_t'\Big\|_{\mathcal V(\theta_0)}\|  g_t -  \hat g_t  \|_{\mathcal V(\theta_0)}
$$
where $\nabla_\theta \phi_t'(x,\theta)=(0,1,X_t/2\exp(-x/2),|X_t|/2\exp(-x/2))'$ and $c:=\min_{\mathcal V(\theta_0)} \alpha/(1-\beta)$. Remark that $\E[\log^+ \|\sup_{x\ge c }\nabla_\theta \phi_t' \|_{\mathcal V(\theta_0)}]<\infty$ because $\E[\log^+ |X_t|]<\infty$. Thus $$\sum_{t=0}^\infty \P\Big(\Big\|\sup_{x\ge c }\nabla_\theta \phi_t\Big\|_{\mathcal V(\theta_0)}\ge \rho^{-t}\Big)<\infty$$
for any $\rho>1$
and then $\|\Phi_t'(0)-\hat\Phi_t'(0)\|\xrightarrow{e.a.s.}0$  by an application of the Borel-Cantelli Lemma because $\|  g_t -  \hat g_t  \|_{\mathcal V(\theta_0)}\xrightarrow{e.a.s.}0$.\\

To check the second condition of the assumption {\bf S'}, we remark that  as $\phi_t$ is twice continuously differentiable, we also have
$$
\Lambda(\Phi_t'-\hat\Phi_t')\le \|\phi_t'(\hat g_t,\cdot)-\phi_t'( g_t,\cdot)\|_{\mathcal V(\theta_0)}\le \|\sup_{x\ge c }\phi_t''\|_{\mathcal V(\theta_0)}\|  g_t -  \hat g_t  \|_{\mathcal V(\theta_0)}
$$
where $\phi_t''(x,\theta)=4^{-1}(\gamma X_t+\delta|X_t|)\exp(-x/2)$. That $\|\Phi_t'(0)-\hat\Phi_t'(0)\|\xrightarrow{e.a.s.}0$ follows again by an application of the Borel-Cantelli Lemma because $\|  g_t -  \hat g_t  \|_{\mathcal V(\theta_0)}\xrightarrow{e.a.s.}0$ and   $\E[\log^+\|\sup_{x\ge c }\phi_t''\|_{\mathcal V(\theta_0)}]<\infty$ as $\E[\log^+ |X_t|]<\infty$.\\

The existence of the stationary solution $(\nabla g_t)$ of the SRE driven by $(\Phi_t')$ follows by an application of Theorem \ref{th:pert}. It also provides that $(\|\nabla g_t \|_{\mathcal V(\theta_0)})$ satisfies the property {\bf (EAS)} that will be useful to derive the existence of the second derivatives below. This stationary solution $(\nabla g_t)$ is continuous as it is the locally uniform asymptotic law of $\nabla \hat g_t$ that is continuous by construction. \\

Deriving a second time and keeping the same notation as above, we obtain another linear SRE satisfied by $(\mathbb H \hat g_{t+1} (\theta))$ for any $1\le i,j\le d$
\begin{multline}\label{sreddg}
\mathbb H  \hat g_{t+1} (\theta)= \phi_t'(\hat g_t(\theta),\theta)\mathbb H \hat g_t(\theta)+ \phi_t''(\hat g_t(\theta),\theta)\nabla  \hat g_t(\theta)\nabla   \hat g_t(\theta)^T \\
+\nabla_\theta \phi_t'(\hat g_t(\theta),\theta)\nabla \hat g_t(\theta)^T+\nabla \hat g_t(\theta)\nabla_\theta \phi_t'(\hat g_t(\theta),\theta)^T+\mathbb H_{\theta } \phi_t(\hat g_t(\theta),\theta).
\end{multline}
For the EGARCH(1,1) model, $\mathbb H_{\theta } \phi_t(x,\theta)$ is identically null. Thus we consider the perturbed SRE
$$
\hat\Phi_t''(h):=\phi_t''( \hat g_t ,\cdot)h+ \phi_t''(  \hat g_t ,\cdot)\nabla   \hat g_t \nabla    g_t^T  
+\nabla_\theta \phi_t'(\hat  g_t ,\cdot)\nabla  g_t^T +\nabla \hat   g_t\nabla_\theta \phi_t'(\hat  g_t ,\cdot)^T, \qquad t\ge 0.
$$
We can apply Theorem \ref{th:pert} on this perturbed SRE and the corresponding stationary SRE
$$
\Phi_t''(h):=\phi_t''( g_t ,\cdot)h+ \phi_t''(  g_t ,\cdot)\nabla   g_t \nabla    g_t^T  
+\nabla_\theta \phi_t'( g_t ,\cdot)\nabla  g_t^T +\nabla   g_t\nabla_\theta \phi_t'( g_t ,\cdot)^T,\qquad t\in\Z.
$$
The details of the proof are omitted as they are similar than those used above on the first derivative. Note that the property {\bf (EAS)} on  $(\|\nabla g_t \|_{\mathcal V(\theta_0)})$ is required to check the first condition of the assumption {\bf S} of Theorems  \ref{th:pert} applied to $\Phi_t''(0)$.

\subsection{Proof of Lemma \ref{lem:appr}}

Let us fix $\mathcal V(\theta_0)$ as in the proof of Lemma \ref{lem:der} such that $g_t$ is twice continuously invertible on $\mathcal V(\theta_0)$ and such that the uniform  log moments on the derivatives exist. Because $\|\tilde\theta_t-\theta_0\|\xrightarrow{a.s.}0$ there exists some random integer $M\ge 1$ such that $\tilde \theta_t\in \mathcal V(\theta_0)$ for any $t\ge M$. Consider the SRE 
$$
\nabla g_{t+1} (\theta)= \phi_t'( g_t(\theta),\theta)\nabla  g_t(\theta)+\nabla_{\theta} \phi_t( g_t(\theta),\theta)\quad \forall t\in \Z
$$
where 
\begin{align*}
\phi_t'(x,\theta)&=\beta-2^{-1}(\gamma Z_t+\delta|Z_t|)\exp(-(x-g_t(\theta_0))/2),\\\nabla_{\theta} \phi_t(x,\theta)&=(1,x,Z_t\exp(-(x-g_t(\theta_0))/2),|Z_t|\exp(-(x-g_t(\theta_0))/2))'.
\end{align*}
As $x\to \exp(-x/2)$ is a Lipschitz continuous function for $x\ge c$, as $\mathcal V(\theta_0)$ is a compact set there exists some $C>0$ such that for all $x\in\mathcal K$, all $\theta\in\mathcal V(\theta_0)$ we have
$$
|\phi_t'(x,\theta)-\phi_t'(g_t(\theta_0),\theta_0)|+\|\nabla_{\theta} \phi_t( x,\theta)-\nabla_{\theta} \phi_t( g_t(\theta_0),\theta_0)\|\le C|Z_t|( \|\theta-\theta_0\|+|x-g_t(\theta_0)|).
$$
Thus,
for any $n\ge t\ge M$, denoting $v_t(\tilde\theta_n)=\|\nabla g_{t+1}(\hat \theta_n)-\nabla g_{t+1}(\theta_0)\|$ we obtain
\begin{align*}
v_{t+1}(\tilde\theta_n)\le& |\phi_t'(g_t(\theta_0),\theta_0)| v_t(\tilde\theta_n)|+\|\nabla g_t(\hat\theta_n)\||\phi_t'(g_t(\theta_0),\theta_0)-\phi_t'(\hat g_t(\hat \theta_n),\hat \theta_n)|\\
&+\|\nabla_{\theta} \phi_t(\hat g_t(\hat \theta_n), \hat \theta_n)-\nabla_{\theta} \phi_t( g_t(\theta_0),\theta_0)\|\\
\le&  |\phi_t'(g_t(\theta_0),\theta_0)| v_t(\tilde\theta_n)+(\|\nabla g_t\|_{\mathcal V(\theta_0)}+1)C|Z_t|( \|\tilde\theta_n-\theta_0\|+|g_t(\tilde \theta_n)-g_t(\theta_0)|).
\end{align*}
By a recursive argument, we obtain for all $n\ge t\ge M$ 
\begin{multline*}
v_t(\tilde \theta_n)\le \sum_{j=M}^{t-1}\prod_{i=j+1}^{t-1}|\phi_i'(g_i(\theta_0),\theta_0)| (\|\nabla g_t\|_{\mathcal V(\theta_0)}+1)C|Z_t|(\|\tilde\theta_n-\theta_0\|+|g_t(\tilde \theta_n)-g_t(\theta_0)|)\\
+\prod_{i=M}^{t-1}|\phi_i'(g_i(\theta_0),\theta_0)| v_M.
\end{multline*}
That $\prod_{i=j+1}^{t-1}|\phi_i'(g_i(\theta_0),\theta_0)|\xrightarrow{e.a.s.}0$ follows  from the assumption {\bf (MM)}   $\E\phi_i'(g_i(\theta_0),\theta_0)^2<1$ and by using the subadditive ergodic theorem of \cite{kingman:1973} on the logarithms. Thus the last term of the sum converges e.a.s to 0 and the corresponding  Cesaro mean $c_n=n^{-1}\sum_{t=M}^n\prod_{i=M}^{t-1}|\phi_i'(g_i(\theta_0),\theta_0)| v_M$ also converges e.a.s. to $0$.\\

Let us treat the term
$$
\sum_{j=M}^{t-1}\prod_{i=j+1}^{t-1}|\phi_i'(g_i(\theta_0),\theta_0)| (\|\nabla g_t\|_{\mathcal V(\theta_0)}+1)C|Z_t|\|\tilde\theta_n-\theta_0\|.
$$
It is a.s. smaller than 
$$
\|\tilde\theta_n-\theta_0\|\sum_{j=M}^{\infty}\prod_{i=j+1}^{t-1}|\phi_i'(g_i(\theta_0),\theta_0)| (\|\nabla g_t\|_{\mathcal V(\theta_0)}+1)C|Z_t|.
$$
This series converges a.s.  because $(\|\nabla g_t \|_{\mathcal V(\theta_0)})$ satisfies {\bf (EAS)} by an application of Theorem \ref{th:pert}. Thus we obtain  that there exist some random variable $a>0$ such that
$$
n^{-1}\sum_{t=M}^n\sum_{j=M}^{t-1}\prod_{i=j+1}^{t-1}|\phi_i'(g_i(\theta_0),\theta_0)| (\|\nabla g_t\|_{\mathcal V(\theta_0)}+1)C|Z_t|\|\tilde\theta_n-\theta_0\|\le a\|\tilde\theta_n-\theta_0\|.
$$\

Finally, the reminding term of the upper bound is
$$
\sum_{j=M}^{t-1}\prod_{i=j+1}^{t-1}|\phi_i'(g_i(\theta_0),\theta_0)| (\|\nabla g_t\|_{\mathcal V(\theta_0)}+1)C|Z_t||g_t(\tilde \theta_n)-g_t(\theta_0)|.
$$
We treat it as in the proof of Theorem \ref{pr:vf}. Its uniform norm converges a.s.  by an application of similar arguments than above. Thus there exists a random continuous function $b$  satisfying $b(\theta_0)=0$ and
\begin{align*}
n^{-1}\sum_{t=M}^n\sum_{j=M}^{t-1}\prod_{i=j+1}^{t-1}|\phi_i'(g_i(\theta_0),\theta_0)|& (\|\nabla g_t\|_{\mathcal V(\theta_0)}+1)C|Z_t||g_t(\tilde \theta_n)-g_t(\theta_0)|\\&\le \sum_{j=M}^{\infty}\prod_{i=j+1}^{t-1}|\phi_i'(g_i(\theta_0),\theta_0)| (\|\nabla g_t\|_{\mathcal V(\theta_0)}+1)C|Z_t||g_t(\tilde \theta_n)-g_t(\theta_0)|\\
&\le b(\tilde\theta_n).
\end{align*}\

Finally we obtain that $n^{-1}\sum_{t=M}^nv_t(\tilde \theta_n)\le a\|\tilde \theta_n-\theta_0\|+b(\tilde \theta_n)+c_n$. Using this bound and the SRE satisfied by the differences $\mathbb H g_{t}(\tilde \theta_n)-\mathbb H g_t(\theta_0)$, we obtain following similar arguments than above that 
$$
\frac 1n\sum_{t=M}^n\|\mathbb H g_t(\tilde \theta_n)-\mathbb H g_t(\theta_0)\|\le a'\|\tilde \theta_n-\theta_0\|+b'(\tilde \theta_n)+c'_n,
$$
where  $a'$ is a positive r.v., $b'$ is a random continuous function satisfying $b'(\theta_0)=0$ and $c'_n\xrightarrow{e.a.s.}0$. We conclude that, conditionally on any possible value of $M=m$, we have  
$$n^{-1}\sum_{t=m}^n\|\mathbb H g_t(\tilde \theta_n)-\mathbb H g_t(\theta_0)\|\xrightarrow{a.s.}0.
$$
Thus, we obtain that 
\begin{equation}\label{eqas}
\P\Big(\lim_{n\to\infty}\frac1n\sum_{t=M}^n\|\mathbb H g_{t}(\tilde \theta_n)-\mathbb H g_{t}(\theta_0)\|=0 \Big)=1.
\end{equation}

It remains to estimate  
$$
\P\Big(\lim_{n\to\infty}\frac1n\sum_{t=1}^M\|\mathbb H g_{t}(\tilde \theta_n)-\mathbb H g_{t}(\theta_0)\|=0\Big)= \sum_{k=1}^\infty \P\Big(\lim_{n\to\infty}\frac1n\sum_{t=1}^k\|\mathbb H g_{t}(\tilde \theta_n)-\mathbb H g_{t}(\theta_0)\|=0 \Big)
 \P(M=k).
$$By continuity of the second derivative $\mathbb H g_t$  we have that $\mathbb H g_{t} (\tilde \theta_n)\to \mathbb H g_{t} (  \theta_0)$ a.s. We deduce that for  any $ k\ge 1$ 
$$\P\Big(\lim_{n\to\infty}\frac1n\sum_{t=1}^k\|\mathbb H g_{t}(\tilde \theta_n)-\mathbb H g_{t}(\theta_0)\|=0 \Big)=1.
$$
The desired result follows easily combining these two last equations with \eqref{eqas}.

\subsubsection*{Acknowledgments} I would like to thank C. Francq and J.-M. Zako\"ian for helpful discussions and for  pointing out some mistakes in an earlier version.

\end{document}